\documentclass[12pt]{amsart}
\usepackage{amsfonts}
\usepackage{amsfonts}
\usepackage{amsfonts}
\usepackage{amsfonts}

\usepackage{mathptmx} 
\usepackage[shortlabels]{enumitem}

\usepackage{enumitem}

\usepackage{amsfonts}
\usepackage{amsmath, amsrefs, amsthm, amssymb}
\usepackage[active]{srcltx}		
\usepackage{pdfsync}					
\usepackage{graphicx}
\usepackage{bbm}
\usepackage{amsfonts}
\usepackage{amsthm}
\usepackage{graphicx}
\usepackage{framed}
\usepackage{multicol,multienum}
\usepackage{graphicx} 
\usepackage{booktabs} 
\usepackage{mathrsfs}
\usepackage{tcolorbox}
\usepackage{bm}
\usepackage{subfigure}
\usepackage{epstopdf}

\DeclareMathOperator\supp{\text{supp}}

\newcommand{\sihao}{\fontsize{14pt}{\baselineskip}\selectfont}

\newtheorem{theorem}{Theorem}[section]

\newtheorem{lemma}[theorem]{Lemma}
\newtheorem{corollary}[theorem]{Corollary}

\theoremstyle{definition}
\newtheorem{definition}[theorem]{Definition}

\newtheorem{remark}[theorem]{Remark}

\numberwithin{equation}{section}

\begin{document}

\title [Non-isospectral two-component cubic Camassa-Holm system]{\sihao{ {\rmfamily Global-in-time solvability and blow-up for a non-isospectral two-component cubic Camassa-Holm system in a critical Besov space}}}

\author{Lei Zhang}
\address{School of Mathematics and Statistics, Hubei Key Laboratory of Engineering Modeling  and Scientific Computing, Huazhong University of Science and Technology  Wuhan 430074, Hubei, P.R. China.}
\email{lei\_zhang@hust.edu.cn (L. Zhang)}

\author{Zhijun Qiao}
\address{School of Mathematical and Statistical Sciences, The University of Texas Rio Grande Valley, Edinburg, TX78539, USA}
\email{zhijun.qiao@utrgv.edu (Z. Qiao)}


\date{\today}

\begin{abstract}
In this paper, we prove the global Hadamard well-posedness of strong solutions to a non-isospectral two-component  cubic Camassa-Holm system in the critical Besov space $B_{2,1}^{\frac{1}{2}}(\mathbb{T})$. 
Our results shows 
that in comparison with 
the well-known work for classic Camassa-Holm-type equations, the existence of global solution only 
relies on the $L^1$-integrability of the variable coefficients $\alpha(t)$ and $\gamma(t)$, but nothing to do with 
the shape or smoothness of the initial data.
The key ingredient of the proof hinges on the careful analysis of the mutual effect among 
two component forms, the uniform bound of approximate solutions, and several crucial estimates of cubic nonlinearities in low-regularity Besov spaces via the Littlewood-Paley decomposition theory. {
A reduced case in our results yields 
the global existence of solutions 
in a Besov space for two kinds of well-known isospectral peakon system with weakly dissipative terms.} Moreover, we derive two kinds of precise blow-up criteria for a strong solution in both critical and non-critical Besov spaces, as well as 
providing specific characterization for the lower bound of the blow-up time, which implies the global existence with additional conditions on the time-dependent parameters $\alpha(t)$ an $\gamma(t)$.
\end{abstract}

\maketitle

\section{Introduction and main results}

The Camassa-Holm (CH) equation
 \begin{equation}\label{CH}
m_t+um_x+2u_xm=0,\quad m=u-u_{xx}
\end{equation}
was derived by Camassa and Holm \cite{1} by approximating directly in the Hamiltonian for Euler equations in the shallow water regime, which has attracted much attention among the communities of the nonlinear systems in recent years.  It was proposed as a model for the unidirectional propagation of the shallow water waves over a flat bottom \cite{2,3}, where $u(t, x)$ stands for the fluid velocity at time $t$ in the spatial $x$ direction. The CH equation is completely integrable with a Bi-Hamiltonian structure and has an infinite number of conservation laws \cite{1,5}. Geometrically, the CH equation describes the geodesic flows on the diffeomorphism group of the unit torus under right-invariant $H^1$ metric \cite{22,Kolev25}
and has algebro-geometric solutions on a symplectic submanifold \cite{445}. One of the remarkable features for the CH equation is that it has the peaked soliton (peakon) solutions in the form of $u(x,t)=c e^{-|x-ct|}$, $c\in \mathbb{R}$, which retain their shapes after interacting with other peakons. Another notable property 
of the CH equation is that the blow-up phenomena occurs only in the form of breaking waves, that is, the wave profile remains bounded while its slope becomes unbounded in a finite time \cite{28,18,21}.

The amazing scenario  
in the CH equation has enhanced 
the search for various CH-type equations with high order nonlinearity. One of the most concerned is the following cubic CH equation, i.e., the FORQ/MCH equation 
\begin{equation}\label{FORQ/MCH}
m_t+\left[(u^2-u_x^2)m\right]_x=0,\quad m=u-u_{xx},
\end{equation}
which was investigated 
independently by Fokas \cite{446}, Olver and Rosenau \cite{13} and Qiao
\cite{35}.  It is shown that the FORQ/MCH equation \eqref{FORQ/MCH} has symmetry properties and hodograph transformation related to other integrable equations  \cite{12} and a Bi-Hamiltonian structure with conservation laws \cite{13} and is completely integrable in the sense of Lax pair \cite{35,39}. Hence, in spirit, it can be solved by the inverse scattering transform method.

As a natural extension of the cubic FORQ/MCH equation \eqref{FORQ/MCH}, Song, Qu and Qiao \cite{52} proposed the following two-component cubic
Camassa-Holm system (called the SQQ system for short):
\begin{equation} \label{SQQ}
\begin{cases}
m_t+[(u-u_x)(v+v_x)m]_x=0,\\
n_t+[(u-u_x)(v+v_x)n]_x=0,\\
m=u-u_{xx},\quad n=v-v_{xx},
\end{cases}	
\end{equation}
which is reduced 
to the FORQ/MCH equation \eqref{FORQ/MCH} when $u=v$. The SQQ system \eqref{SQQ} is proven to possess infinitely many conservation laws,  Bi-Hamiltonian structure,  Lax formulation and multi-peakon solutions. 
Moreover, the SQQ system \eqref{SQQ}  is geometrically integrable since it describes the pseudospherical surfaces.
The aforementioned equations belong to the isospectral category because 
 the spectral parameter in the Lax pair is independent of the time variable. Motivated by the work done by Beals, Sattinger and Szmigielski \cite{beals1999multi,beals2000multipeakons}, Chang, Chen and Hu \cite{chang2014generalized} recently generalized the CH equation to a non-isospectral CH equation through 
 the classic determinant technique. It is shown that the non-isospectral CH equation still remains to be integrable in the sense of 
 a non-isospectral Lax pair. Moreover, it admits multi-peakon solutions similar to the CH equation. Subsequently, by using the similar approach 
 in \cite{chang2014generalized}, Chang, Hu and Li \cite{60} studied 
 the following generalized non-isospectral two-component  cubic Camassa-Holm system (called the 2NSQQ system for short):
\begin{equation} \label{2NSQQ}
\begin{cases}
m_t+(\rho m)_x=0,\\
n_t+(\rho n)_x=0,\\
\rho_x=(\alpha+\gamma)m(v+v_x)-\alpha n(u-u_x),\\
m=u-u_{xx},\quad n=v-v_{xx},
\end{cases}	
\end{equation}
where $\alpha, \gamma $ are two arbitrary time-dependent parameters. As a special case, if one takes $\alpha\equiv1$ and $\gamma\equiv 0$, then the system \eqref{2NSQQ} reduces to the SQQ equation \eqref{SQQ}. Furthermore, if one sets $u=v$ in \eqref{2NSQQ}, then the system \eqref{2NSQQ} reduces to the FORQ/MCH equation \eqref{FORQ/MCH} via the SQQ system \eqref{SQQ}. Similar to the non-isospectral CH equation, 
the 2NSQQ system \eqref{2NSQQ} also has multi-peakon solutions. Moreover, it is integrable in the sense of a non-isospectral Lax pair, namely, the system \eqref{2NSQQ} can be obtained by the compatibility condition of the following two linear systems
\begin{equation*}
\begin{array}{l}
\left(                 
  \begin{array}{ccc}   
    \psi_1 \cr  
    \psi_2 \cr  
  \end{array}
\right)_x  
\end{array}=\frac{1}{2}U\begin{array}{l}
\left(                 
  \begin{array}{ccc}   
    \psi_1 \cr  
    \psi_2 \cr  
  \end{array}
\right),~  
\end{array}\quad \begin{array}{l}
\left(                 
  \begin{array}{ccc}   
    \psi_1 \cr  
    \psi_2 \cr  
  \end{array}
\right)_t  
\end{array}=\frac{1}{2}V\begin{array}{l}
\left(                 
  \begin{array}{ccc}   
    \psi_1 \cr  
    \psi_2 \cr  
  \end{array}
\right), 
\end{array}
\end{equation*}
where the two $2\times 2$ matrixes $U,V$ are given by
\begin{equation}
\begin{array}{l}
U=\left(                 
  \begin{array}{ccc}   
    -1 & \lambda m\cr  
    -\lambda n & 1\cr  
  \end{array}
\right), \nonumber         
\end{array}
\end{equation}
and
\begin{equation}
\begin{array}{l}
 V=\left(                 
  \begin{array}{ccc}   
    (2\alpha+\gamma)\lambda^{-2}+\rho & -2\alpha\lambda^{-1}(u-u_x)-\lambda m\rho\cr  
    2(\alpha+\gamma)\lambda^{-1}(v+v_x)+\lambda n \rho & -(2\alpha+\gamma)\lambda^{-2}-\rho\cr  
  \end{array}
\right),\nonumber         
\end{array}
\end{equation}
and the spectral parameter $\lambda (t) $ satisfies the following differential equation
\begin{equation*}
\begin{split}
\dot{\lambda} (t) =\frac{\gamma(t)}{\lambda (t)}.
\end{split}
\end{equation*}

From the intrinsic structure of the system \eqref{2NSQQ}, one knows that
the transport terms in the system  are no longer local since the first two equations can be presented by the integral of unknowns via the third equation in \eqref{2NSQQ}. The other important feature different from the classic CH-type equations is that the coefficients of the 2NSQQ system \eqref{2NSQQ} are time-dependent parameters. It is natural to ask if 
these features have serious effects on the development of solutions to the system \eqref{2NSQQ}, 
such as the well-posedness, blow-up phenomena, and long time asymptotic behavior. Minor part of those questions was answered in our very recent paper \cite{zhang2020periodic}, where 
we established the local-in-time well-posedness to the 2NSQQ system in the non-critical Besov spaces $B_{2,r}^s(\mathbb{T}) $ with $s>1/2$, $1\leq r\leq\infty$, and some specific blow-up criteria are also addressed 
with appropriate conditions on the initial data.

The goal of the present paper is two-fold. First, we want to understand how the time-dependent variable coefficients and nonlocal structure of transport terms can affect the well-posedness of the 2NSQQ system \eqref{2NSQQ} in the critical Besov space $B_{2,1}^{1/2}(\mathbb{T})$. For the case of the CH equation, it is shown in \cite{constantin1998well,constantin1998global} that the existence of global solution is closely related to the shape of the initial data, but not the smoothness or size, such as the sign condition for the initial momentum density. The similar phenomena has also been studied 
for the Degasperis-Procesi (DP) equation \cite{liu2006global}, the Novikov equation \cite{wuyin2012well}, the two-component Camassa-Holm system 
\cite{guan2010global} and so on. For the case of the FORQ/MCH equation \eqref{FORQ/MCH} with cubic nonlinearity, it is proved that even if the initial momentum density does not change sign, the solutions to \eqref{FORQ/MCH} can still blow up in a finite time 
 \cite{fuguiliuqu2013cauchy}. Thereby few global existence results for the equations \eqref{FORQ/MCH} and \eqref{SQQ} already existed  
 in the literature. Unlike the isospectral CH-type equations as we mentioned earlier in this paper, the surprising but amazing thing is: for any given initial data, we can establish the global Hadamard well-posedness for the non-isospectral 2NSQQ system \eqref{2NSQQ} in the critical Besov space. Instead of assuming structural conditions on the initial data, only is the matter of the $L^1$-integrability of the time-dependent parameters $\alpha$ and $\gamma$ 
  required to promote the existence of global solution. Within our best knowledge, this phenomena does not appear in any isospectral CH-type equations.

To state 
our main results, let us  define the pseudo-differential operator $\partial_x^{-1}$ as mentioned in \cite{zhang2020periodic}:
\begin{equation*}
\begin{split}
\partial_x^{-1}f(x)\doteq \int_0^xf(y)dy-x\int_\mathbb{T}f(y)dy-\int_\mathbb{T} \left[\int_0^xf(y)dy-x\int_\mathbb{T}f(y)dy\right]dx.
\end{split}
\end{equation*}
The representation of the operator $\partial_x^{-1}$ in the Fourier domain is formulated by
\begin{equation*}
\widehat{\partial_x^{-1}f}(n)=
\begin{cases}
\frac{1}{in}\widehat{f}(n), &\mbox{if}~n\neq0,\\
0, &\mbox{if}~n=0.
\end{cases}	
\end{equation*}
In the context, 
the notation $\overline{f}(t)=\int_\mathbb{T}f(t,x)dx$ will frequently be used.
Then, the periodic Cauchy problem for the 2NSQQ system \eqref{2NSQQ} can be reformulated into the following nonlocal transport equations on $\mathbb{T}\doteq\mathbb{R}/\mathbb{Z}$:
\begin{equation}\label{P2NSQQ}
\begin{cases}
m_t+\rho m_x=-m\left(\psi(t,x)-\overline{\psi}(t)\right),&~~ t>0,~x\in\mathbb{T},\\

n_t+\rho n_x=-n\left(\psi(t,x)-\overline{\psi}(t)\right),&~~ t>0,~x\in\mathbb{T},\\

\rho=\partial_x^{-1}\psi,&~~  t>0,~x\in\mathbb{T},\\

m(0,x)=m_0(x),\quad n(0,x)=n_0(x),&~~x\in \mathbb{T},
\end{cases}
\end{equation}
with
\begin{equation*}
\begin{split}
\psi(t,x)\doteq(\alpha+\gamma)(v+v_x)m-\alpha(u-u_x)n.
\end{split}
\end{equation*}

\begin{definition}
For any $s\in \mathbb{R}$, $1\leq r \leq \infty$, let us define 
the following three spaces
$$
X_{s,r}\doteq B_{2,r}^{s}(\mathbb{T})\times B_{2,r}^{s}(\mathbb{T}), \quad E_{2,1}^s(\infty)\doteq \bigcap_{T>0} E_{2,1}^s(T),
$$
and
$$
E_{2,1}^s(T)\doteq C([0,T];B_{2,1}^{s}(\mathbb{T})\times B_{2,1}^{s}(\mathbb{T}))\bigcap C^1([0,T];B_{2,1}^{s-1}(\mathbb{T})\times B_{2,1}^{s-1}(\mathbb{T})),
$$
for any finite $T>0$.
\end{definition}

The result of the local and global well-posedness for the periodic Cauchy problem \eqref{P2NSQQ} in the critical Besov spaces may now be enunciated by the following theorem.

\begin{theorem} [Hadamard Well-posedness] \label{theorem1}
Assume the initial data $(m_0, n_0)\in  X_{1/2,1}$.
\begin{itemize}[leftmargin=0.75cm]
\item [(1)] (Local result) If the time-dependent parameters $\alpha(\cdot),\gamma(\cdot) \in L^1_{loc}([0,\infty);\mathbb{R})$, then there is a finite time $T^*>0$ with the following condition 
\begin{equation*}
\begin{split}
\int_0^{T^*}(|\alpha(t)|+|\gamma(t)|)dt \leq \frac{\ln 2}{12C^3\hbar^2\Big(\|m_0\|_{B^{1/2}_{2,1}}+\|n_0\|_{B^{1/2}_{2,1}}\Big)}\doteq K(m_0,n_0),
\end{split}
\end{equation*}
where $\hbar(\cdot):  \mathbb{R}^+\mapsto \mathbb{R}^+$ is a modulus of continuity
\begin{equation*}
\begin{split}
\hbar(x)=& \left(x+8 C^3 x^3 \int_0^{T^*}(|\alpha(t)|+|\gamma(t)|)dt\right)\exp{\left\{4 C^3 x^2\int_0^{T^*}(|\alpha(t)|+|\gamma(t)|)dt\right\}},
\end{split}
\end{equation*}
such that the Cauchy problem \eqref{P2NSQQ} has a unique solution $(m,n)\in E_{2,1}^{1/2}(T^*)$,
and the data-to-solution map $\Lambda (\cdot):(m_0,n_0)\rightarrow (m,n)$ is H\"{o}lder continuous from $X_{1/2,1}$ into $E_{2,1}^{1/2}(T^*)$.

\item [(2)] (Global result) If the time-dependent parameters $\alpha(\cdot),\gamma (\cdot) \in L^1([0,\infty);\mathbb{R})$ satisfy the following bound
$$
\int_0^{\infty}(|\alpha(t)|+|\gamma(t)|)dt\leq \widetilde{K}(m_0,n_0),
$$
where $\widetilde{K}(m_0,n_0)$ is defined by replacing $\hbar(x)$ of $K(m_0,n_0)$ with
\begin{equation*}
\begin{split}
\widetilde{\hbar}(x)=& \left(x+8 C^3 x^3 \int_0^{\infty}(|\alpha(t)|+|\gamma(t)|)dt\right)\exp{\left\{4 C^3 x^2\int_0^{\infty}(|\alpha(t)|+|\gamma(t)|)dt\right\}},
\end{split}
\end{equation*}
for all $x\geq0$, then there exists a global solution $(m,n)$ to the Cauchy problem \eqref{P2NSQQ}.
\end{itemize}
\end{theorem}

Regarding the above Theorem 1.2, we have several remarks listed below.
\begin{remark}
\begin{itemize}[leftmargin=0.62cm]
\item [(i)] The existence of global strong solution for isospectral CH-type equations has been established in several papers, such as Constantin \cite{constantin1998global,constantin1998well}, Liu \cite{liu2006global}, and Liu and Yin \cite{liuyin2006global}. 
     Those results share the common characteristic that the solutions globally exist 
     only when there are appropriate structural conditions imposed on the initial data. For example, in \cite{constantin1998well}, Constantin proposed the shape condition
    $y_0(x)=u_0(x)-\partial_xu_0(x)\geq 0$ for any $x\in \mathbb{T}$ and proved the global-in-time solvability of the CH equation in the Sobolev space $H^1(\mathbb{T})$.
     In contrast to the existing work done in the literature \cite{constantin1998global,constantin1998well, liu2006global, liuyin2006global}, our results in Theorem 1.1 show that the existence of global-in-time solution does not depends on the specific shape of the initial data, but on the $L^1$-integrability of the time-dependent parameters $\alpha(t)$ and $\gamma(t)$ in the 2NSQQ system. Within our 
      knowledge, this phenomena is brand 
      new, which tells us 
      that the well-posedness scenario of the non-isospectral CH system  \eqref{P2NSQQ} has a significant 
      difference from the isospectral CH-type equations.

\item [(ii)] The proof of Theorem 1.1 is involved in more delicate techniques  in comparison with the FORQ equation in \cite{fuguiliuqu2013cauchy}, 
    the SQQ system in \cite{yan2015qualitative}, and the non-critical case in \cite{zhang2020periodic}. Let us explain details below.
\begin{itemize}

    \item [(a)]  In the critical case, due to the low-regularity of the space $B_{2,1}^{-1/2}$, the crucial bilinear estimate applied in \cite{yan2015qualitative} and \cite{zhang2020periodic} is inapplicable for the proof of strong convergence of the approximate solutions $(m_k,n_k)_{k\geq1}$ in $X_{-1/2,1}$. To cope with 
        this difficulty, we shall first take a step back to prove the strong convergence in the larger Besov space $X_{-1/2,\infty}$ with the help of 
        well-chosen endpoint bilinear estimates, the Logarithmic-type interpolation inequality, and the Osgood lemma. Then, employing 
        an interpolation argument leads the strong convergence to be lifted into the space $X_{-1/2,1}$.

  \item [(b)] Unlike the isospectral CH-type equations \eqref{CH}-\eqref{SQQ}, the uniform bound for the approximate solutions $(m_k,n_k)_{k\geq1}$ nonlinearly depends on the time variable, that is,  the $t$-variable is involved in the integral $A(0,t)\doteq\int_0^t(|\alpha(t')|+|\gamma(t')|)dt'$. As a consequence, the widely used approach for the CH-type equations is no longer working 
      for the present case. Employing a different iterative method, we obtain 
      the uniform bound 
      by virtue of a modulus of continuity $\hbar(x)$, which is closely related to the $L^1$-integrability of the parameters $\alpha(t),\gamma(t)$. Such a uniform bound seems caught for the first time. The advantage of the $t$-nonlinear-dependence property of the uniform bound lies on: 
      for a given initial data, the lifespan of solution can be extended to infinity by imposing proper integrability conditions to the parameters $\alpha(t),\gamma(t)$.

  \item  [(c)] The third main difficulty comes from the mutual effect between two components and the analysis of high order nonlinearities in the system, such as the estimates for $\int_\mathbb{T}\phi_{k,j} dx$ in (3.25). Instead of applying the $L^\infty$-estimate used in the derivation of uniform bound, one should carefully utilize the structure of nonlinear terms. To be more precise, it is worth  to point out 
      that the Schwartz space $\mathcal {S}(\mathbb{T})$ is dense in $B_{2,1}^{3/2}(\mathbb{T})$, and $u_{k+j}-\partial_x u_{k+j} \in B_{2,1}^{3/2}(\mathbb{T})$, $n_{k+j}-n_k\in B_{2,\infty}^{-3/2}(\mathbb{T})\hookrightarrow \mathcal {S}'(\mathbb{T})$. It follows from the nonhomogeneous dyadic blocks $(\Delta_l)_{l\geq-1}$ that
    $$
    \left\langle u_{k+j}-\partial_x u_{k+j},n_{k+j}-n_k\right\rangle=\sum_{|l-l'|\leq 2}\left\langle\Delta_{l}\left(u_{k+j}-\partial_x u_{k+j}\right), \Delta_{l'}(n_{k+j}-n_k)\right\rangle.
    $$
  Then, we can derive from the above identity a bound in terms of the norm $\|n_{k+j}-n_k\|_{B_{2,1}^{-1/2}}$, which is crucial in the proof of the convergence. The other term in the integral can be treated by a similar manner.
  \end{itemize}
\end{itemize}
\end{remark}

{
\begin{remark}
As we mentioned before, the non-isospectral system \eqref{2NSQQ} can be reduced to two important isospectral CH-type equations by appropriately 
 choosing  parameters $\alpha(t)$ and $\gamma(t)$, namely,
\begin{itemize}[leftmargin=0.75cm]
\item [$\bullet$] SQQ system \eqref{SQQ} (when $\alpha\equiv1$, $\gamma\equiv 0$);

\item [$\bullet$] FORQ/MCH equation \eqref{FORQ/MCH} (when $\alpha\equiv1$, $\gamma\equiv 0$ and $u\equiv v$).

\end{itemize}

It follows from  Theorem \ref{theorem1} (1) that the blow-up time $T^*$ can not be extended to infinity, which implies
the local-in-time well-posedness results for the SQQ system \eqref{SQQ} and the FORQ/MCH equation \eqref{FORQ/MCH} respectively. Therefore Theorem \ref{theorem1} covers the local results in \cite{fuguiliuqu2013cauchy,yan2015qualitative}.
\end{remark}


Another interesting thing from our Theorem \ref{theorem1} is that one can regain 
the global-in-time existence  for the SQQ system \eqref{SQQ} and the FORQ/MCH equation \eqref{FORQ/MCH} with the damping perturbation (also called the weakly dissipative term). More precisely, let us  consider the following damping perturbation of the FORQ/MCH equation:
\begin{equation*}
(\lambda-\mbox{FORQ/MCH})\quad \quad
\begin{cases}
m_t+\left[(u^2-u_x^2)m\right]_x+\lambda m=0,&~~ t>0,~x\in\mathbb{T},\\
m(0,x)=m_0(x),&~~x\in\mathbb{T},
\end{cases}
\end{equation*}
where $\lambda>0$ is the dissipative parameter. The $\lambda$-FORQ/MCH equation is actually a special case of the 2NSQQ system \eqref{P2NSQQ}. Indeed, let 
 $\widetilde{m}(t,x)=e^{2\lambda t}m(t,x)$, $\widetilde{u}(t,x)=e^{2\lambda t}u(t,x)$, then 
 apparently $\widetilde{m}(t,x)$ and $\widetilde{u}(x,t)$ satisfy the following parameterized FORQ/MCH equation:
\begin{equation*}
\begin{cases}
\widetilde{m}_t+\left[e^{-2\lambda t}(\widetilde{u}^2-\widetilde{u}_x^2)\widetilde{m}\right]_x=0,&~~ t>0,~x\in\mathbb{T},\\
\widetilde{m}(0,x)=m_0(x),&~~x\in\mathbb{T},
\end{cases}
\end{equation*}
which is obviously the 2NSQQ system \eqref{P2NSQQ} with $\alpha(t)=e^{-2\lambda t}$, $\gamma\equiv0$ and $u\equiv v$. 
Casting  Theorem \ref{theorem1} and the fact  $\int_0^\infty e^{-2\lambda t}dt=\frac{1}{2\lambda}$ immediately yields the following result.

\begin{corollary} \label{corollary1}
Assume the initial data $m_0\in B_{2,1}^{1/2}(\mathbb{T})$, and the dissipative parameter $\lambda$ satisfies the following inequality
$$
\lambda\geq \frac{6C^3}{\ln 2} \left(\|m_0\|_{B^{1/2}_{2,1}}+\frac{4C^3}{\lambda}\|m_0\|_{B^{1/2}_{2,1}}^3\right)^2e^{4C^3\|m_0\|_{B^{1/2}_{2,1}}^2/\lambda}.
$$
Then, the $\lambda$-FORQ/MCH equation has a unique global strong solution $m\in C([0,T];B_{2,1}^{1/2}(\mathbb{T}))$.
\end{corollary}

Similarly,  
we can also have 
the following Cauchy problem with the damping perturbation for the SQQ system:
\begin{equation*}
(\lambda-\mbox{SQQ})\quad \quad
\begin{cases}
m_t+[(u-u_x)(v+v_x)m]_x+\lambda m=0,&~~ t>0,~x\in\mathbb{T},\\
n_t+[(u-u_x)(v+v_x)n]_x+\lambda n=0,&~~ t>0,~x\in\mathbb{T},\\
m(0,x)=m_0(x),~~n(0,x)=n_0(x),&~x\in\mathbb{T},
\end{cases}	
\end{equation*}
where $\lambda>0$ is the dissipative parameter. Letting 
 $\widetilde{m}(t,x)=e^{2\lambda t}m(t,x)$, $\widetilde{n}(t,x)=e^{2\lambda t}n(t,x)$ sends 
 the functions $\widetilde{m}(t,x)$ and $\widetilde{u}(x,t)$ to 
 the following parameterized SQQ system:
\begin{equation*}
\begin{cases}
\widetilde{m}_t+\left[e^{-2\lambda t}(\widetilde{u}-\widetilde{u}_x)(\widetilde{v}+\widetilde{v}_x)\widetilde{m}\right]_x=0,&~~ t>0,~x\in\mathbb{T},\\
\widetilde{n}_t+\left[e^{-2\lambda t}(\widetilde{u}-\widetilde{u}_x)(\widetilde{v}+\widetilde{v}_x)\widetilde{n}\right]_x=0,&~~ t>0,~x\in\mathbb{T},\\
\widetilde{m}(0,x)=m_0(x),~~\widetilde{n}(0,x)=n_0(x),&~x\in\mathbb{T},
\end{cases}	
\end{equation*}
which is actually the 2NSQQ system \eqref{P2NSQQ} with $\alpha(t)=e^{-2\lambda t}$ and $\gamma\equiv0$. Utilizing 
Theorem \ref{theorem1} again generates 
the following Corollary.

\begin{corollary} \label{corollary2}
Assume the initial data $(m_0, n_0)\in  X_{1/2,1}$, and the dissipative parameter $\lambda$ satisfies the following inequality
$$
\lambda\geq \frac{6C^3}{\ln 2} \left(\|m_0\|_{B^{1/2}_{2,1}}+\|n_0\|_{B^{1/2}_{2,1}}+\frac{4C^3}{\lambda}(\|m_0\|_{B^{1/2}_{2,1}}+\|n_0\|_{B^{1/2}_{2,1}})^3\right)^2e^{4C^3(\|m_0\|_{B^{1/2}_{2,1}}+\|n_0\|_{B^{1/2}_{2,1}})^2/\lambda}.
$$
Then, the $\lambda$-SQQ system has a unique global strong solution $(m,n)\in E_{2,1}^{1/2}(\infty)$.
\end{corollary}

It is worth to note that the two Corollaries \ref{corollary1}-\ref{corollary2} are new phenomena for the weakly dissipative shallow water wave equations, which indicate that the dissipative parameters can also determine the global existence results independent of 
the shape conditions on initial data (cf. \cite{wu2008blow,wu2009global}).
}
Our second goal 
in the present paper is to investigate the finite time blow-up regime 
for the 2NSQQ system \eqref{P2NSQQ} in Besov spaces, which in some sense tell us how the time-dependent parameters $\alpha(t)$ and $\gamma(t)$ affect the singularity formation. The first blow-up criteria is formulated for the $X_{1/2,1}$-valued initial data given in the following theorem.

\begin{theorem}\label{theorem2}
Assume the parameters $\alpha,\gamma \in L^1_{loc}([0,\infty);\mathbb{R})$, and the initial data $(m_0,n_0)\in X_{1/2,1}$.  If the  corresponding solution  $(m,n)$ blows up in the finite time $T^*$, then
\begin{equation*}
\begin{split}
 \int_0^{T^*}(|\alpha(t')|+|\gamma(t')|)\Big( \|m(t')\|_{\dot{B}_{\infty,1}^0}^2+\|n(t')\|_{\dot{B}_{\infty,1}^0}^2\Big)dt'=\infty,
\end{split}
\end{equation*}
and the blow-up time $T^*$ is estimated as follows 
\begin{equation*}
\begin{split}
T^*\geq T(m_0,n_0)\doteq \sup_{t>0}\left\{\int_0^{t}(|\alpha(t')|+|\gamma(t')|)dt'\leq\frac{1}{C\Big(\| m_0\|_{B_{2,1}^{1/2}}+ \| n_0\|_{B_{2,1}^{1/2}}\Big)^2}\right\}.
\end{split}
\end{equation*}
\end{theorem}

\begin{remark}
The blow-up criteria stated in Theorem \ref{theorem2} is not contradictory with Theorem \ref{theorem1}, which reveals 
that the $L^1$-integrability of the parameters $\alpha$ and $\gamma$ determine the behavior of the solutions to the 2NSQQ system \eqref{2NSQQ}. Indeed, given a 
further assumption $\|\alpha\|_{L^1(0,\infty)}+\|\gamma\|_{L^1(0,\infty)}= 1/2C(\| m_0\|_{B_{2,1}^{1/2}}+ \| n_0\|_{B_{2,1}^{1/2}})^2$, then it follows from  Theorem \ref{theorem2} that $T^*=T(m_0,n_0)=\infty$, that is to say, the solution $(m,n)$ exists globally. Moreover, due to $H^s(\mathbb{T})\cong   B_{2,2}^{s}(\mathbb{T})\hookrightarrow B_{2,1}^{1/2}(\mathbb{T})$ with $s>1/2$,  Theorem \ref{theorem2} improved the blow-up criteria in Sobolev spaces (cf. \cite{
zhang2020periodic}).
\end{remark}

Our third goal in the paper is to construct 
the blow-up criteria with the initial data in a bit regular space $X_{1/2+\epsilon,r}$, for any $\epsilon \in (0,1/2)$. The following theorem presents this result.

\begin{theorem}\label{theorem3}
Let $\epsilon \in (0,1/2)$ and $r\in [1,\infty]$. Assume  the parameters $\alpha,\gamma \in L^1_{loc}([0,\infty);\mathbb{R})$, and the initial data $(m_0,n_0)\in  X_{1/2+\epsilon,r}$. If the corresponding solution  $(m,n)$ blows up in the finite time $T^*$, then
\begin{equation*}
\begin{split}
 \int_0^{T^*} (|\alpha(t')|+|\gamma(t')|)\Big( \|m(t')\|_{\dot{B}_{\infty,2}^{0}}^2+\|n(t')\|_{\dot{B}_{\infty,2}^{0}}^2\Big)dt'=\infty,
\end{split}
\end{equation*}
and the blow-up time $T^*$ is estimated as follows 
\begin{equation*}
\begin{split}
T^*\geq T'(m_0,n_0)\doteq\sup_{t>0}\left\{\int_0^t(|\alpha(t')|+|\gamma(t')|)dt'\leq \frac{1}{C\Big(\sqrt{2}e+\| m_0\|_{B_{2,r}^{1/2+\epsilon}}+\| n_0\|_{B_{2,r}^{1/2+\epsilon}} \Big)^6}\right\}.
\end{split}
\end{equation*}
\end{theorem}

\begin{remark}
The blow-up regime described in Theorem \ref{theorem3} is resided in the Besove space $\dot{B}^0_{\infty,2}(\mathbb{T})$. Recall the Sobolev embedding
$$
\dot{B}_{\infty,1}^{0}(\mathbb{T})\hookrightarrow\dot{B}_{\infty,2}^{0}(\mathbb{T})\hookrightarrow\dot{F}_{\infty,2}^{0}(\mathbb{T})\cong BMO \hookrightarrow\dot{B}_{\infty,\infty}^{0}(\mathbb{T}),
$$
where $\dot{F}_{\infty,2}^{0}(\mathbb{T})$ is a special Lizorkin-Triebel space and $BMO$ is the space of the bounded mean oscillation (cf. \cite{bergh2012interpolation,strichartz1980bounded}).  This is slightly stronger than the blow-up criteria in Theorem \ref{theorem2}, but weaker than the one  for the CH equation in $\dot{B}_{\infty,\infty}^{0}(\mathbb{R})$ \cite{danchin2001few} and the one for the Fornberg-Whitham
equation in the $BMO$ case \cite{longwei2018wave}. Such gap originates from the balance between the order of the nonlinearity and the order of the interpolation inequality in \eqref{6.12}. More specifically, the order of nonlinearity of the 2NSQQ system \eqref{P2NSQQ} is cubic instead of quadratic, which leads to the quadratic $L^\infty$-estimation in Eq. (6.18) below. Note that the similar blow-up criteria has also been derived for the harmonic heat flow equation onto a sphere \cite{kozono2002critical}.
\end{remark}

The remaining of this paper is organized as follows. In Section 2, we introduce the Littlewood-Paley theory, and some well-known results of the transport theory in Besov spaces. The Sections 3-5 are devoted to the proof of the local and global Hadamard well-posedness for the 2NSQQ system \eqref{2NSQQ} in the critical Besov space. In Section 6, we derive two kinds of blow-up criteria for the strong solutions to the 2NSQQ system \eqref{2NSQQ}. The lower bound of the blow-up time are also addressed.

\vskip3mm\noindent
\textbf{Notation.} Throughout the paper, since all spaces of functions 
  are over the torus $\mathbb{T}$, we will drop $\mathbb{T}$ in the notation of function spaces if there is no any specific to be clarified. 
  Let $1\leq p \leq \infty$ and $X$ be a Banach space, for simplicity, 
  the functional spaces $L^p(0, T; X)$  and $C^k([0,T];X)$ are denoted  by $L^p_T(X)$ and by $C^k_T(X)$, respectively.

\section{Preliminaries}\label{sec:prelim}
In this section, we recall some well-known facts of the Littlewood-Paley decomposition theory and the linear transport theory in Besov spaces.
\begin{lemma} [\cite{64,chemin2004localization}]\label{lem:Dyaunit}
 Denote by $\mathcal {C}$ the annulus of centre 0, short radius $3/4$ and long radius $8/3$. Then there exits two positive radial functions $\chi$ and $\varphi$ belonging respectively to $C_c^\infty(B(0,4/3))$ and $ C_c^\infty(\mathcal {C})$ such that
 $$
 \chi(\xi)+\sum_{q\geq0}\varphi(2^{-q}\xi)=1, \quad \forall \xi \in \mathbb{R}^d,
 $$
 $$
 |p-q|\geq2\Rightarrow \supp \varphi(2^{-q}\cdot)\cap \supp \varphi(2^{-p}\cdot)=\emptyset,
 $$
 $$
 q\geq1\Rightarrow\supp \chi ( \cdot)\cap \supp \varphi (2^{-q}\cdot)=\emptyset,
 $$
and
 $$
 \frac{1}{3}\leq \chi^2(\xi)+\sum_{q\geq0} \varphi^2(2^{-q}\xi)\leq1,\quad \forall \xi\in \mathbb{R}^d.
 $$
\end{lemma}

The Fourier transformation of $u$ on the $d$-dimension torus $\mathbb{T}^d$ is defined by
$$
\widehat{u}(\alpha)\doteq \int_{\mathbb{T}^d}u(x)e^{-2\pi i\langle\alpha,x\rangle }dx.
$$
Let 
$$
\check{h}(x)=\sum_{\alpha\in\mathbb{Z}^d} \chi(\alpha)e^{2\pi i\langle\alpha,x\rangle }\quad \mbox{and}\quad h_q(x)=\sum_{\alpha\in\mathbb{Z}^d} \varphi (2^{-q}\alpha)e^{2\pi i\langle\alpha,x\rangle },
$$
then  the nonhomogeneous dyadic blocks $(\Delta_q)_{q\geq-1}$  can be defined as follows
$$
\Delta_qu\doteq0,\quad q\leq-2,
$$
$$
\Delta_{-1}u\doteq \sum_{\alpha \in \mathbb{Z}^d} \chi (\alpha) \widehat{u}(\alpha)e^{ 2\pi i\langle\alpha,x\rangle }=\int_{\mathbb{T}^d} u(x-y)\check{h}(y)dy,\quad  q=-1,
$$
$$
\Delta_{q}u\doteq \sum_{\alpha \in \mathbb{Z}^d} \varphi (2^{-q}\alpha)\widehat{u}(\alpha) e^{ 2\pi i\langle\alpha,x\rangle }=\int_{\mathbb{T}^d} u(x-y)h_q(y)dy,\quad  q\geq 0.
$$
The nonhomogeneous Littlewood-Paley  decomposition of $u\in \mathscr{S}'(\mathbb{T}^d)$ is denoted by 
$$
u=\sum_{q\geq-1} \Delta_{q} u.
$$
The high-frequency cut-off operator is referred to 
$$
S_qu\doteq\sum\limits_{p\leq q-1 }\Delta _p u, \quad \forall q\in\mathbb{N}.
$$

\begin{definition} [\cite{64
}]\label{def:PBesov}
For any $s\in \mathbb{R}$ and $p,r\in [1,\infty] $, the $d$-dimension nonhomogeneous Besov space $B_{p,r}^s(\mathbb{T}^d) $ is defined by
$$
B_{p,r}^s(\mathbb{T}^d)  \doteq\left\{u\in \mathscr{S}'(\mathbb{T}^d); ~\|u\|_{B_{p,r}^s}=\|(2^{qs}\| \Delta _qu\|_{L^p})_{l\geq -1}\|_{l^r}  <\infty\right\},
$$
If $s=\infty$,  $B_{p,r}^\infty(\mathbb{T}^d)\doteq\bigcap_{s\in\mathbb{R}}B_{p,r}^s(\mathbb{T}^d)$.
\end{definition}

Using Lemma \ref{lem:Dyaunit}, one can also define the homogeneous dyadic blocks $(\dot{\Delta}_q)_{q\in \mathbb{Z}}$ and the homogeneous cut-off operators $\dot{S}_q$ as follows:
$$
\dot{\Delta}_{q}u\doteq \sum_{\alpha \in \mathbb{Z}^d} \varphi (2^{-q}\alpha)\widehat{u}(\alpha) e^{ 2\pi i\langle\alpha,x\rangle }=\int_{\mathbb{T}^d} u(x-y)h_q(y)dy,\quad  \forall q\in \mathbb{Z},
$$
and
$$
\dot{S}_qu\doteq \sum_{\alpha \in \mathbb{Z}^d} \chi (\alpha) \widehat{u}(\alpha)e^{ 2\pi i\langle\alpha,x\rangle }=\int_{\mathbb{T}^d} u(x-y)\check{h}(y)dy,\quad  \forall q\in \mathbb{Z}.
$$

\begin{definition} [\cite{64}] \label{def:HPBesov}
For any $s\in \mathbb{R}$ and $p,r\in [1,\infty] $, the $d$-dimension inhomogeneous Besov space $\dot{B}_{p,r}^s(\mathbb{T}^d) $ is defined by
$$
\dot{B}_{p,r}^s(\mathbb{T}^d) \doteq\left\{u\in \mathscr{S}'(\mathbb{T}^d)/\mathscr{P}(\mathbb{T}^d); ~\|u\|_{\dot{B}_{p,r}^s}\doteq\left\|(2^{qs}\left\|\dot{\Delta}_qu\right\|_{L^p})_{l\in \mathbb{Z}}\right\|_{l^r} <\infty\right\},
$$
where $\mathscr{P}(\mathbb{T}^d)$ denotes the space of the polynomial functions on $\mathbb{T}^d$.
\end{definition}

Now, let us list 
some useful results in the transport equation theory in Besov spaces,  which are crucial to the proofs of our main theorems.
\begin{lemma}[\cite{64,danchin2005fourier}]\label{lem:priorieti}
Assume that $p,r\in [1,\infty]$ and  $s > -\frac{d}{p}$. Let $v$ be a vector field such that
 $\nabla v$ belongs to $L^1_T(B_{p,r}^{s-1})$ if $s>1+\frac{d}{p}$ or to $ L^1_T(B_{p,r}^{\frac{d}{p}}\cap L^\infty)$ otherwise. Suppose also that $f_0\in B_{p,r}^s$, $F\in L^1_T(B_{p,r}^s)$ and that  $f \in L^\infty_T(B_{p,r}^s)\cap C_T(\mathcal {S}')$ solves the linear transport equation
 \begin{eqnarray*}\label{eq:transeq}
(T)\quad\begin{cases}
\partial_tf+v\cdot\nabla f=F,\\
f|_{t=0}=f_0.
\end{cases}
\end{eqnarray*}
Then there is a constant $C$ depending only on $s,p$ and $r$ such that the following statements hold:

1) If $r=1$ or $s\neq 1+\frac{d}{p}$,  then
\begin{equation*}
\begin{split}
\|f(t)\|_{B_{p,r}^s}\leq  \|f_0\|_{B_{p,r}^s}+\int_0^t \|F(\tau)\|_{B_{p,r}^s}d\tau+C\int_0^t V'(\tau)\|f(\tau)\|_{B_{p,r}^s}d \tau,
\end{split}
\end{equation*}
or
\begin{equation}\label{2.1}
\begin{split}
\|f(t)\|_{B_{p,r}^s}\leq  e^{CV(t)}\left(\|f_0\|_{B_{p,r}^s}+\int_0^t e^{-CV(\tau)}\|F(\tau)\|_{B_{p,r}^s}d\tau\right)
\end{split}
\end{equation}
hold, where $ V(t)= \int_0^t\|\nabla v (\tau)\|_{B_{p,r}^{\frac{d}{p}}\bigcap L^\infty}d\tau$ if $s<1+\frac{d}{p}$ and $
V(t)=\int_0^t\|\nabla v(\tau)\|_{B_{p,r}^{s-1}}d\tau$ otherwise.

2) If $s\leq\frac{d}{p}$ and, in addition, $\nabla f_0\in L^\infty $, $\nabla f\in L^\infty _T(L^\infty)$ and $\nabla F\in L^1 _T(L^\infty)$, then
\begin{equation*}
\begin{split}
&\|f(t)\|_{B_{2,r}^s}+\|\nabla f(t)\|_{L^\infty}\\
&\quad \leq  e^{CV(t)}\left(\|f_0\|_{B_{2,r}^s}+\|\nabla f_0\|_{L^\infty}+\int_0^t e^{-CV(\tau)}(\|F(\tau)\|_{B_{2,r}^s}+\|\nabla F(\tau)\|_{L^\infty})d\tau\right).
\end{split}
\end{equation*}

3) If $f=v$, then for all $s>0$, the estimate \eqref{2.1} holds with $V(t)=\int_0^t\|\nabla v(\tau)\|_{L^\infty}d\tau$.

4) If $r<\infty$, then $f\in C_T(B_{2,r}^{s})$. If $r=\infty$, then $f\in C_T(B_{2,1}^{s'})$ for all $s'<s$.
\end{lemma}

\begin{lemma} [\cite{64,danchin2005fourier}] \label{lem:transwell}
Let $(p,p_1,r)\in [1,\infty]^3$. Assume that $s>-d\min\{\frac{1}{p_1},\frac{1}{p'}\}$ with $p'=(1-\frac{1}{p})^{-1}$. Let $f_0\in B_{p,r}^s$ and $F\in L^1_T(B_{p,r}^s)$. Let $v \in L^\rho_T( B^{-M}_{\infty,\infty})$ for some $\rho > 1$, $M > 0$ and $\nabla v \in L^1_T(B^{\frac{d}{p_1}}_{p_1,\infty}\bigcap L^\infty)$ if $s < 1+\frac{d}{p_1}$, and $\nabla v  \in L^1_T( B^{s-1}_{p_1,r})$ if $s >  1+\frac{d}{p_1}$ or $s = 1+\frac{d}{p_1}$ and $r =1$. Then $(T)$ has a unique solution $f\in L^\infty_T( B^{s}_{p,r})\bigcap (\bigcap_{s'<s} C_T(B^{s'}_{p,1})$ and the inequalities in Lemma \ref{lem:priorieti} hold true. If, moreover, $r<\infty$, then we have $f\in  C_T(B^{s}_{p,r})$.
\end{lemma}

\section{Proof of Theorem \ref{theorem1}for Existence}
We divide the proof of the existence of solutions to  (1.5) into the following several steps.
\subsection{Approximate solution}
We construct the approximate solutions via the Friedrichs iterative method. Starting from $(m_1,n_1)\doteq (S_1 m_0,S_1n_0)$, we recursively define a sequence of functions
$(m_k,n_k)_{k\geq1}$ by solving the following linear nonlocal transport equations
\begin{equation}\label{3.1}
\begin{cases}
\partial_t m_{k+1}+\rho_k \partial_x m_{k+1}=-m_k\left(\psi_k(t,x)-\overline{\psi_k}(t)\right),&~~ t>0,~x\in\mathbb{T},\\

\partial_tn_{k+1}+\rho_k \partial_x n_{k+1}=-n_k\left(\psi_k(t,x)-\overline{\psi_k}(t)\right),&~~ t>0,~x\in\mathbb{T},\\

m_{k+1}(x,0)=S_{k+1}m_0(x),&~~ t>0,~x\in\mathbb{T},\\
n_{k+1}(x,0)=S_{k+1}n_0(x),&~~ ~x\in\mathbb{T},
\end{cases}
\end{equation}
where the transport velocity is given by $\rho_k=\partial_x^{-1}\psi_k$, and
$$
\psi_k(t,x)=(\alpha+\gamma)(v_k+\partial_x v_k)m_k-\alpha(u_k-\partial_x u_k)n_k.
$$
It follows from the definition of frequency truncation operator that $(S_{k+1}m_0,S_{k+1}n_0)\in \bigcap _{s\in \mathbb{R}}X_{s,1}$.
Assume by induction that,  given $k\in \mathbb{N}$ and $T>0$, the approximate solution $(m_k,n_k)\in L^\infty_T(X_{1/2,1})$. Since the Besov space $B_{2,1}^{1/2}$ is a Banach algebra, we have
\begin{equation}\label{3.2}
\begin{split}
\int_0^T\|\psi_k(t)\|_{B_{2,1}^{1/2}}dt\leq& \int_0^T(|\alpha(t)|+|\gamma(t)|)\|v_k+\partial_x v_k\|_{B_{2,1}^{1/2}}\|m_k\|_{B_{2,1}^{1/2}}dt\\
&+\int_0^T|\alpha(t)|\|u_k-\partial_x u_k\|_{B_{2,1}^{1/2}}\|n_k\|_{B_{2,1}^{1/2}}dt\\
\leq&\sup_{t\in [0,T]}\Big(\|m_k(t)\|_{B_{2,1}^{1/2}}^2+\|n_k(t)\|_{B_{2,1}^{1/2}}^2\Big)\int_0^T(|\alpha(t)|+|\gamma(t)|)dt< \infty,
\end{split}
\end{equation}
where we have used  $\alpha, \gamma\in L^1_{loc}([0,\infty);\mathbb{R})$. Since the operator $(1-\partial_x^2)^{-1}$ is a $S^{-2}$-multiplier (cf. Proposition 2.78 in \cite{64}), we have for any $k\geq 1$
\begin{equation}\label{3.3}
\begin{split}
\|u_k\|_{B_{2,1}^{5/2}}=\|(1-\partial_x^2)^{-1}m_k\|_{B_{2,1}^{5/2}}\approx \|m_k\|_{B_{2,1}^{1/2}}, \quad \|v_k\|_{B_{2,1}^{5/2}}=\|(1-\partial_x^2)^{-1}n_k\|_{B_{2,1}^{5/2}}\approx \|n_k\|_{B_{2,1}^{1/2}} .
\end{split}
\end{equation}
By \eqref{3.3}, one can verify that the functions $m_k \psi_k$ and $n_k \psi_k$ belong to  $L^1_T(B_{2,1}^{1/2})$. For $\overline{\psi_k}(t)$, it follows from the properties of the Littlewood-Paley decomposition operators
\begin{equation}\label{3.4}
\begin{split}
\left\| \overline{\psi_k}(t)\right\|_{B_{2,1}^{1/2}}&=\sum\limits_{q\geq-1}2^{q/2}\left\|\Delta_q\overline{\psi_k}(t)\right\|_{L^2}
\\
&=2^{-1/2}\left\|\Delta_{-1}\overline{\psi_k}(t)\right\|_{L^2} \leq2^{-1/2}\left\|\overline{\psi_k}(t)\right\|_{L^2} \leq 2^{-1/2} \|\psi_k(t)\|_{L^\infty}.
\end{split}
\end{equation}
Moreover, by using the Sobolev embedding  $B_{2,1}^{1/2}\hookrightarrow  L^\infty $, we get
\begin{equation}\label{3.5}
\begin{split}
\|\psi_k(t)\|_{L^\infty} \leq C(|\alpha(t)|+|\gamma(t)|)\|m_k\|_{L^\infty}\|n_k\|_{L^\infty} \leq C (|\alpha(t)|+|\gamma(t)|)\|m_k(t) \|_{B_{2,1}^{1/2}}\|n_k(t) \|_{B_{2,1}^{1/2}}.
\end{split}
\end{equation}
From the estimates \eqref{3.4}, \eqref{3.5} and the algebraic property of $B_{2,1}^{1/2}$, we get
\begin{equation*}
\begin{split}
\int_0^T \|m_k(t) \overline{\psi_k}(t)\|_{B_{2,1}^{1/2}}dt&\leq 2^{-1/2}\int_0^T \|m_k(t) \|_{B_{2,1}^{1/2}} \|\psi_k(t)\|_{L^\infty}dt\\
& \leq C \int_0^T  (|\alpha(t)|+|\gamma(t)|)\|m_k(t) \|_{B_{2,1}^{1/2}}^2\|n_k(t) \|_{B_{2,1}^{1/2}}dt\\
&\leq C \|m_k\|_{L^\infty_T(B^{1/2}_{2,1})}^2\|n_k\|_{L^\infty_T(B^{1/2}_{2,1})}\int_0^T  (|\alpha(t)|+|\gamma(t)|) dt <\infty,
\end{split}
\end{equation*}
which implies that $m_k  \overline{\psi_k}\in L^1_T(B^{1/2}_{2,1}) $. Similarly we also have $n_k  \overline{\psi_k}\in L^1_T(B^{1/2}_{2,1}) $. Thereby the nonlinear terms on the right hand side of system (3.1) belongs to $L^1_T(B^{1/2}_{2,1}) $.

Furthermore, since $\partial_x\rho_k=\psi_k(t,x)-\overline{\psi_k}(t)$, one can deduce from  the previous estimates that $\partial_x\rho_k \in L^1_T(B^{1/2}_{2,\infty}\cap L^\infty)$. Thanks to Lemma \ref{lem:transwell},  the system \eqref{3.1} has a unique solution $(m_{k+1},n_{k+1}) \in C_T(X_{1/2,1})$.

\subsection{Uniform bound}

For any $k\geq1$, we define
\begin{equation*}
\begin{split}
F_0\doteq\|m_0\|_{B^{1/2}_{2,1}}+\|n_0\|_{B^{1/2}_{2,1}},\quad F_k(t)\doteq \|m_{k}(t)\|_{B^{1/2}_{2,1}}+\|n_{k}(t)\|_{B^{1/2}_{2,1}},
\end{split}
\end{equation*}
and
\begin{equation*}
\begin{split}
A(s,t)\doteq\int_s^t (|\alpha(t')|+|\gamma(t')|)dt',\quad  \forall t\geq s \geq 0.
\end{split}
\end{equation*}

By applying   Lemma \ref{lem:priorieti} to the system (3.1) with respect to $m_k$ and using the fact of $\partial_x\rho_k=\psi_k(t,x)-\overline{\psi_k}(t)$, we get
\begin{equation}\label{3.6}
\begin{split}
 \|m_{k+1}(t)\|_{B^{1/2}_{2,1}}
&\leq \exp{\left\{C\int_0^t\|\psi_k(t',x)-\overline{\psi_k}(t')\|_{B^{1/2}_{2,1} }dt'\right\}}\| m_0\|_{B^{1/2}_{2,1}} \\
&+ \int_0^t  \exp{\left\{C\int_{t'}^t\|\psi_k(\tau,x)-\overline{\psi_k}(\tau)\|_{B^{1/2}_{2,1} }d\tau\right\}} \left\|m_k\left(\psi_k(t',\cdot)-\overline{\psi_k}(t')\right)\right\|_{B^{1/2}_{2,1}}dt'.
\end{split}
\end{equation}
Using the estimates \eqref{3.2}-\eqref{3.5}, we have
\begin{equation*}
\begin{split}
\int_0^t\|\psi_k(t',x)-\overline{\psi_k}(t')\|_{B^{1/2}_{2,1}}ds\leq C \int_0^t(|\alpha(t')|+|\gamma(t')|)\|m_k(t') \|_{B_{2,1}^{1/2}}\|n_k(t') \|_{B_{2,1}^{1/2}}dt',
\end{split}
\end{equation*}
and
\begin{equation*}
\begin{split}
\|m_k\left(\psi_k(t,x)-\overline{\psi_k}(t)\right)\|_{B^{1/2}_{2,1}} \leq C(|\alpha(t)|+|\gamma(t)|)\|m_k(t) \|_{B_{2,1}^{1/2}}^2\|n_k(t) \|_{B_{2,1}^{1/2}}.
\end{split}
\end{equation*}
Inserting the last two estimates into \eqref{3.6}, we get
\begin{equation}\label{3.7}
\begin{split}
\|m_{k+1}(t)\|_{B^{1/2}_{2,1}}
\leq  &\exp{\left\{CV_k(0,t)\right\}}\| m_0\|_{B^{1/2}_{2,1}} \\
&+ C\int_0^t \exp{\left\{CV_k(t',t)\right\}} (|\alpha(t')|+|\gamma(t')|)\|m_k(t')\|_{B_{2,1}^{1/2}}^2\|n_k(t')\|_{B_{2,1}^{1/2}} dt'
\end{split}
\end{equation}
with
\begin{equation*}
\begin{split}
V_k(s,t)\doteq\int_s^t(|\alpha(t')|+|\gamma(t')|)\|m_k(t')\|_{B_{2,1}^{1/2}}\|n_k(t')\|_{B_{2,1}^{1/2}}dt'.
\end{split}
\end{equation*}
Using the similar procedure to \eqref{3.1} associated with $n_{k+1}$, one can deduce that
\begin{equation}\label{3.8}
\begin{split}
\|n_{k+1}(t)\|_{B^{1/2}_{2,1}}
\leq  &\exp{\left\{CV_k(0,t)\right\}}\| n_0\|_{B^{1/2}_{2,1}} \\
&+ C\int_0^t \exp{\left\{CV_k(t',t)\right\}} (|\alpha(t')|+|\gamma(t')|)\|m_k(t')\|_{B_{2,1}^{1/2}}\|n_k(t')\|_{B_{2,1}^{1/2}}^2 dt'.
\end{split}
\end{equation}
It then follows from \eqref{3.7} and \eqref{3.8} that
\begin{equation*}
\begin{split}
F_{k+1}(t)
\leq  &\exp{\left\{CV_k(0,t)\right\}}F_0+ C\int_0^t \exp{\left\{CV_k(t',t)\right\}} (|\alpha(t')|+|\gamma(t')|)\\
&\times\Big(\|m_k(t')\|_{B_{2,1}^{1/2}}^2\|n_k(t')\|_{B_{2,1}^{1/2}}+\|m_k(t')\|_{B_{2,1}^{1/2}}
\|n_k(t')\|_{B_{2,1}^{1/2}}^2\Big) dt'\\
\leq&\exp{\left\{CV_k(0,t)\right\}}F_0+ C\int_0^t \exp{\left\{CV_k(t',t)\right\}} (|\alpha(t')|+|\gamma(t')|) F_{k}^3(t') dt',
\end{split}
\end{equation*}
which indicates the following iterative inequality
\begin{equation}\label{3.9}
\begin{split}
F_{k+1}(t)
\leq C\exp{\left\{C\int_0^t(|\alpha(t')|+|\gamma(t')|)F_k^2(t')dt'\right\}}
\left(F_0+\int_0^t (|\alpha(t')|+|\gamma(t')|) F_{k}^3(t') dt'\right).
\end{split}
\end{equation}

Notice that both the integral $V_k(s,t)$ and the iterative inequality \eqref{3.9} both contain an additional factor $|\alpha(t)|+|\gamma(t)|$, so the classic method used in \cite{danchin2001few,fuguiliuqu2013cauchy,yan2015qualitative} is inapplicable in present case. To overcome this difficult, we use another iterative method to derive the uniform bound. Without loss of generality, we  assume that the generic constant $C $ satisfies $C>1$.

For $k=1$,  and a fixed $t_0>0$, it follows from \eqref{3.9} that
\begin{equation}\label{3.10}
\begin{split}
\sup_{t\in [0,t_0]}F_{1}(t)
\leq&C\exp{\left\{C\int_0^{t_0} (|\alpha(t')|+|\gamma(t')|)\left(2C F_0\right)^2 dt'\right\}}\\
&\quad\times\left(F_0+\int_0^{t_0} (|\alpha(t')|+|\gamma(t')|)\left(2C F_0\right)^3 dt'\right)\\
\leq&2C\left(F_0+8 C^3 F_0^3 A(0,t_0)\right)\exp{\left\{4 C^3 F_0^2A(0,t_0)\right\}}
\doteq  2C\hbar(F_0),
\end{split}
\end{equation}
where
\begin{equation*}
\begin{split}
\hbar(x)= \left(x+8 C^3 x^3 A(0,t_0)\right)\exp{\left\{4 C^3 x^2A(0,t_0)\right\}},\quad  \forall x\geq0.
\end{split}
\end{equation*}
It is clear that $\hbar(0)=0$ and the function $\hbar(x) $ is a modulus of continuity defined on $\mathbb{R}^+$, which is  independent of the initial data $(m_0,n_0)$.

For $k=2$, we deduce from \eqref{3.9}  that for any $t\in [0,t_0]$
\begin{equation}\label{3.11}
\begin{split}
F_{2}(t)
\leq&C\exp{\left\{C\int_0^t(|\alpha(t')|+|\gamma(t')|)F_1^2(t')dt'\right\}}
\left(F_0+\int_0^t (|\alpha(t')|+|\gamma(t')|) F_1^3(t') dt'\right)\\
\leq& C\left(F_0+ 8C^3\hbar^3(F_0)A(0,t)\right)\exp{\left\{4C^3\hbar^2(F_0)A(0,t)\right\}}.
\end{split}
\end{equation}
Since the time-dependent functions $\alpha$ and $\gamma$ are locally Lebesgue integrable on $\mathbb{R}^+$, it follows from the absolute continuity of the integral that one can find a time $0<T^*\leq t_0$ such that
\begin{equation}\label{3.12}
\begin{split}
A(0,T^*)&\leq\frac{\ln 2}{12C^3\left(F_0+8 C^3 F_0^3 A(0,T^*)\right)\exp{\left\{4 C^3 F_0^2A(0,T^*)\right\}}}\\
&=\frac{\ln 2}{12C^3\hbar^2(F_0)}.
\end{split}
\end{equation}
Using the fact of $\hbar(F_0)\geq F_0$, we get from \eqref{3.11}-\eqref{3.12} that
\begin{equation}\label{3.13}
\begin{split}
\sup_{t\in [0,T^*]}F_{2}(t)&\leq C\left(F_0+ 8C^3 \hbar^3(F_0)A(0,T^*)\right)\exp{\left\{4C^3\hbar^2(F_0)A(0,T^*)\right\}}\\
&\leq C\hbar(F_0)\left(1+ 8C^3\hbar^2(F_0)A(0,T^*)\right)\exp{\left\{4C^3\hbar^2(F_0)A(0,T^*)\right\}}\\
&\leq C\hbar(F_0)\exp{\left\{12C^3\hbar^2(F_0)A(0,T^*)\right\}}\\
&\leq 2C\hbar(F_0),
\end{split}
\end{equation}
where the second inequality in \eqref{3.13} used the basic estimate $1+x\leq e^x$ for all $x\geq0$. Note that the estimate \eqref{3.10} remains to be true if we replace the time $t_0$ by $T^*$.

Assume inductively that, for any given $k\in \mathbb{N}$, the following estimate holds
\begin{equation*}
\begin{split}
\sup_{t\in [0,T^*]}F_{k}(t) \leq 2C\hbar(F_0).
\end{split}
\end{equation*}
Then for $F_{k+1}(t)$, it follows from \eqref{3.9} and \eqref{3.13} that
\begin{equation*}
\begin{split}
\sup_{t\in [0,T^*]}F_{k+1}(t)
\leq&C\exp{\left\{C\int_0^{T^*}(|\alpha(t')|+|\gamma(t')|)F_k^2(t')dt'\right\}}
\left(F_0+\int_0^{T^*} (|\alpha(t')|+|\gamma(t')|) F_{k}^3(t') dt'\right)\\
\leq& C\left(F_0+8C^3\hbar(F_0)^3 A(0,T^*)\right)\exp{\left\{4C^3\hbar^2(F_0)A(0,T^*)\right\}}\\
\leq& C\left(\hbar(F_0)+8C^3\hbar(F_0)^3 A(0,T^*)\right)\exp{\left\{4C^3\hbar^2(F_0)A(0,T^*)\right\}}\\
 \leq & 2C\hbar(F_0).
\end{split}
\end{equation*}
Using the mathematical induction with respect to $k$, it is easily seen that
\begin{equation*}
\begin{split}
\sup_{t\in [0,T^*]}F_{k}(t) \leq  2C\hbar(F_0)
\end{split}
\end{equation*}
holds for any $k\geq 0$, which implies the uniform bound
\begin{equation}\label{3.14}
\begin{split}
 \sup_{t\in [0,T^*]}\Big(\|m_k(t)\|_{B^{1/2}_{2,1}}+\|n_k(t)\|_{B^{1/2}_{2,1}}\Big)\leq 2C\hbar(F_0),\quad \forall k \geq 0,
\end{split}
\end{equation}

As a consequence, the approximate solutions $(m_k,n_k)_{k\geq1}$ is uniformly bounded in the space $C_{T^*}(X_{1/2,1})$. Moreover, by using the system \eqref{3.1}, one can verify that the sequence $(\partial_tm_k,\partial_tn_k)_{k\geq1}$ is uniformly bounded in $C_{T^*}(X_{-1/2,1})$. Therefore we obtain that $(m_k,n_k)_{k\geq1}$ is uniformly bounded in $ E_{2,1}^{1/2}(T^*)$.

\subsection{Convergence}
We first show that the approximate solutions $(m_k,n_k)_{k\geq1}$ is a Cauchy sequence  in $C_{T^*}(X_{-1/2,\infty})$, and then extend the convergent result to $C_{T^*}(X_{-1/2,1})$ by using an interpolation argument. To this end, we set
\begin{equation*}
\begin{split}
\mathscr{D}_{k,j}(t)\doteq \|(m_{k+j}-m_{k})(t)\|_{B^{-1/2}_{2,\infty}}+\|(n_{k+j}-n_{k})(t)\|_{B^{-1/2}_{2,\infty}},\quad \forall k,j\geq1.
\end{split}
\end{equation*}

\textbf{Claim:} For any $k,j\geq1$, there is a positive constant $C$ independent of $k,j$ such that
\begin{equation}\label{3.15}
\begin{split}
\mathscr{D}_{k+1,j}(t) \leq  e^{C\hbar^2(F_0)}\left(2^{-k} +\int_0^t   (  |\alpha(t')|+|\gamma(t')|)\mathscr{D}_{k,j}(t') \left(1
 +\log\frac{4C\hbar(F_0)}{\mathscr{D}_{k,j}(t')}\right) dt'\right).
\end{split}
\end{equation}
We define for any $k,j\in \mathbb{N}$
$$
\varphi_{k,j}(t,x)\doteq(v_k+\partial_x v_k)(m_k-m_{k+j})+\left[v_k-v_{k+j}+\partial_x (v_k-  v_{k+j})\right]m_{k+j},
$$
and
$$
\phi_{k,j}(t,x)\doteq(u_{k+j}-\partial_x u_{k+j})(n_{k+j}-n_k)+ \left[u_{k+j}-u_k-\partial_x (u_{k+j} - u_k)\right]n_k.
$$
Direct calculation shows that
\begin{equation}\label{3.16}
\begin{split}
\psi_k(t,x)-\psi_{k+j}(t,x) =(\alpha+\gamma) \varphi_{k,j}(t,x)+\alpha\phi_{k,j}(t,x).
\end{split}
\end{equation}
Using above notations, one can derive from \eqref{3.1} and \eqref{3.16} that
\begin{equation}\label{3.17}
\begin{split}
&\partial_t (m_{k+j+1}-m_{k+1})+\rho_{k+j} \partial_x( m_{k+j+1}- m_{k+1})\\
&\quad=(m_k-m_{k+j})\left(\psi_k(t,x)-\overline{\psi_k}(t)\right)+(\alpha+\gamma)m_{k+j}\varphi_{k,j}(t,x)+\alpha m_{k+j}\phi_{k,j}(t,x)\\
&\qquad-\alpha m_{k+j}\overline{\phi_{k,j}(t,x)} -(\alpha+\gamma)m_{k+j}\overline{\varphi_{k,j}(t,x)}-(\rho_{k+j}-\rho_k) \partial_x m_{k+1}\\
&\quad \doteq \mathcal {F}_1(u_k,u_{j+k},v_j,v_{j+k}),
\end{split}
\end{equation}
and
\begin{equation}\label{3.18}
\begin{split}
&\partial_t(n_{k+j+1}-n_{k+1})+\rho_{k+j} \partial_x( n_{k+j+1}- n_{k+1})\\
&\quad =(n_k-n_{k+j})\left(\psi_k(t,x)-\overline{\psi_k}(t)\right)+(\alpha+\gamma)n_{k+j}\varphi_{k,j}(t,x)+\alpha n_{k+j}\phi_{k,j}(t,x)\\
&\qquad-(\alpha+\gamma)n_{k+j}\overline{\varphi_{k,j}(t,x)} -\alpha n_{k+j}\overline{\phi_{k,j}(t,x)}-(\rho_{k+j}-\rho_k) \partial_x n_{k+1}\\
&\quad \doteq \mathcal {F}_2(u_k,u_{j+k},v_j,v_{j+k}).
\end{split}
\end{equation}
Applying Lemma \ref{lem:priorieti} to Eq.\eqref{3.17} leads to
\begin{equation}\label{3.19}
\begin{split}
 &\|(m_{k+j+1}-m_{k+1})(t)\|_{B^{-1/2}_{2,\infty}}\leq \|S_{k+j+1}m_0-S_{k+1}m_0\|_{B^{-1/2}_{2,\infty}}\\
&\qquad +C\int_0^t\|\psi_{k+j}(t',\cdot)-\overline{\psi_{k+j}}(t')\|_{B^{1/2}_{2,\infty}\cap L^\infty}\|(m_{k+j+1}-m_{k+1})(t')\|_{B^{-1/2}_{2,\infty}}dt'\\
 &\qquad +\int_0^t\|\mathcal {F}_1(u_k,u_{j+k},v_j,v_{j+k})\|_{B^{-1/2}_{2,\infty}}dt'.
\end{split}
\end{equation}
According to the definition of the Littlewood-Paley blocks $\Delta_p$ and the almost orthogonal property $\Delta_p\Delta_q=0$ for $|p-q|\geq 2$, we have
\begin{equation}\label{3.20}
\begin{split}
 \|S_{k+j+1}m_0-S_{k+1}m_0\|_{B^{-1/2}_{2,\infty}}& =\sup_{p\geq -1}2^{-p/2}\left\|\Delta_p\sum_{k+1\leq q \leq k+j}\Delta_q m_0\right\|_{L^2}\\
 &\leq\sum_{k\leq p \leq k+j+1}\left(2^{-p}2^{p/2}\sum_{|p-q|\leq1}\left\|\Delta_p\Delta_q m_0\right\|_{L^2}\right)\\
 &\leq C 2^{-k}\sum_{k\leq p \leq k+j+1}2^{p/2} \left\|\Delta_p m_0\right\|_{L^2} \leq C 2^{-k} \|m_0\|_{B_{2,1}^{1/2}} .
\end{split}
\end{equation}
By the estimates \eqref{3.2}, \eqref{3.4} and \eqref{3.5}, we have
\begin{equation}\label{3.21}
\begin{split}
\|\psi_{k+j}(t,\cdot)-\overline{\psi_{k+j}}(t)\|_{B^{1/2}_{2,\infty}\cap L^\infty}
&\leq C(\|\psi_{k+j}(t,\cdot)\|_{B^{1/2}_{2,1}}+\|\psi_{k+j}(t,\cdot)\|_{L^\infty})\\
&\leq C (|\alpha(t)|+|\gamma(t)|)\Big(\|m_{k+j} \|_{B_{2,1}^{1/2}}^2+\|n_{k+j}\|_{B_{2,1}^{1/2}}^2\Big)\\
&\leq C\hbar^2(F_0)(|\alpha(t)|+|\gamma(t)|),
\end{split}
\end{equation}
where the last inequality used the uniform bound for  solutions $(m_k,n_k)$ in $X_{1/2,1}$. Now let us estimate the nonlinear terms involved in $\mathcal {F}_1(u_k,u_{j+k},v_j,v_{j+k})$. To this end, we need the following bilinear estimates in Besov spaces.
\begin{lemma}[Moser estimate \cite{danchin2001few}] \label{moser}
Let
$s_1\leq 1/p<s_2$ $(s_2\geq 1/p$ if $r=1$) and $s_1+s_2>0$, then
\begin{equation*}
\begin{split}
\|fg\|_{B^{s_1}_{p,r}}\leq  C\|f\|_{B^{s_1}_{p,r}}\|g\|_{B^{s_2}_{p,r}}.
\end{split}
\end{equation*}
\end{lemma}
\begin{lemma} [Endpoint bilinear estimate \cite{1}]\label{endpoint}
For any $p\geq2$, the paraproduct is continuous from $B_{p,1}^{-1/p}\times (B_{p,1}^{1/p}\cap L^\infty)$ into $B_{p,1}^{-1/p}$, that is, there is some $C>0$ such that
\begin{equation*}
\begin{split}
\|fg\|_{B^{-1/p}_{p,\infty}}\leq C\|f\|_{B^{-1/p}_{p,1}}\|g\|_{B^{1/p}_{2,\infty}\cap L^\infty}.
\end{split}
\end{equation*}
\end{lemma}
For the first term in $\mathcal {F}_1$, by  using \eqref{3.21} and Lemma \ref{endpoint} with $p=2$, we have
\begin{equation}\label{3.22}
\begin{split}
&\|(m_k-m_{k+j})\left(\psi_k(t,x)-\overline{\psi_k}(t)\right)\|_{B^{-1/2}_{2,\infty}}\\
&\quad \leq C \|m_k-m_{k+j}\|_{B^{-1/2}_{2,1}} \|\psi_k(t,\cdot)-\overline{\psi_k}(t)\|_{B^{1/2}_{2,\infty}\cap L^\infty}\\
&\quad \leq C \hbar^2(F_0)(|\alpha(t)|+|\gamma(t)|)\|m_k-m_{k+j}\|_{B^{-1/2}_{2,1}}.
\end{split}
\end{equation}
For the second term, by suitably choosing $s_1$ and $s_2$ in Lemma \ref{moser}, and using the norm-equivalence \eqref{3.3} as well as the uniform bound for approximate solutions, we have for any $\epsilon\in (0,1)$ that
\begin{eqnarray}\label{3.23}
\begin{split}
\|(\alpha+\gamma)m_{k+j}\varphi_{k,j}\|_{B^{-1/2}_{2,\infty}}&\leq  C(|\alpha(t)|+|\gamma(t)|)\|m_{k+j}\|_{B^{1/2}_{2,\infty}\cap L^\infty}\|\varphi_{k,j}(t,\cdot)\|_{B^{-1/2}_{2,1}}\\
&\leq C\hbar(F_0)(|\alpha(t)|+|\gamma(t)|) \Big(\|(v_k+\partial_x v_k)(m_k-m_{k+j})\|_{B^{-1/2}_{2,1}}\\
&\quad +\|[v_k-v_{k+j}+\partial_x (v_k-  v_{k+j})]m_{k+j}\|_{B^{-1/2+\epsilon}_{2,1}}\Big)\\
&\leq C\hbar(F_0)(|\alpha(t)|+|\gamma(t)|) \Big(\|v_k+\partial_x v_k\|_{B^{1/2+\epsilon}_{2,1}}\|m_k-m_{k+j}\|_{B^{-1/2}_{2,1}}\\
&\quad +\|[v_k-v_{k+j}+\partial_x (v_k-  v_{k+j})]\|_{B^{-1/2+\epsilon}_{2,1}}\|m_{k+j}\|_{B^{1/2}_{2,1}}\Big)\\
&\leq C\hbar^2(F_0)(|\alpha(t)|+|\gamma(t)|) \Big(\|m_k-m_{k+j}\|_{B^{-1/2}_{2,1}} +\|n_k-  n_{k+j}\|_{B^{-1/2}_{2,1}}\Big).
\end{split}
\end{eqnarray}
The third term can be estimated as
\begin{equation}\label{3.24}
\begin{split}
\|\alpha m_{k+j}\phi_{k,j}\|_{B^{-1/2}_{2,\infty}}
&\leq C\hbar(F_0)|\alpha(t)|\Big(\|(u_{k+j}-\partial_x u_{k+j})(n_{k+j}-n_k)\|_{B^{-1/2}_{2,1}}\\
&\quad + \|[u_{k+j}-u_k-\partial_x (u_{k+j} - u_k)]n_k\|_{B^{-1/2+\epsilon}_{2,1}}\Big)\\
&\leq C\hbar(F_0)|\alpha(t)|\Big[\Big(\|u_{k+j}\|_{B^{1/2+\epsilon}_{2,1}}+\|u_{k+j}\|_{B^{3/2+\epsilon}_{2,1}}\Big)\|n_{k+j}-n_k\|_{B^{-1/2}_{2,1}}\\
&\quad+  \Big(\|u_{k+j}-u_k\|_{B^{-1/2+\epsilon}_{2,1}}+\|u_{k+j} - u_k\|_{B^{1/2+\epsilon}_{2,1}}\Big)\|n_k\|_{B^{1/2}_{2,1}} \Big] \\
&\leq C\hbar^2(F_0)|\alpha(t)|\Big(\|m_k-m_{k+j}\|_{B^{-1/2}_{2,1}} +\|n_k-  n_{k+j}\|_{B^{-1/2}_{2,1}}\Big).
\end{split}
\end{equation}
For the forth term, since the torus $\mathbb{T}$ is a compact set, it follows from \eqref{3.4} and Lemma \ref{moser} that
\begin{eqnarray}\label{3.25}
\begin{split}
\|\alpha m_{k+j}\overline{\phi_{k,j}}\|_{B^{-1/2}_{2,\infty}} &\leq C|\alpha(t)|\|m_{k+j}\|_{B^{1/2}_{2,\infty}\cap L^\infty}\|\overline{\phi_{k,j}}\|_{B^{-1/2}_{2,1}} \\
&\leq C\hbar(F_0)|\alpha(t)|\left|\int_\mathbb{T}\phi_{k,j}(t,x)ds\right|\\
&\leq C\hbar(F_0)|\alpha(t)| \left( \left|\langle u_{k+j}-\partial_x u_{k+j},n_{k+j}-n_k \rangle\right|+\left|\langle u_{k+j}-u_k-\partial_x (u_{k+j} - u_k),n_k \rangle\right| \right).
\end{split}
\end{eqnarray}
To estimate the last two terms occurring in the right hand side of \eqref{3.25}, we first observe from the uniform bounded  of $(m_k,n_k)_{k\geq -1}$ that $u_{k+j}-\partial_x u_{k+j}\in B_{2,1}^{3/2}$, $n_{k+j}-n_k\in B_{2,1}^{1/2} \hookrightarrow B_{2,\infty}^{-3/2}=(B_{2,1}^{3/2})'$, where $X'$ denotes the duality of the space $X$. Since the Schwartz space $\mathcal {S}$ is dense in $B_{2,1}^{3/2}$, by using a density argument and the Littlewood-Paley decomposition operators $(\Delta_k)_{k\geq -1}$ together with the H\"{o}lder's inequality, we have
\begin{eqnarray}\label{3.26}
\begin{split}
\left|\langle u_{k+j}-\partial_x u_{k+j},n_{k+j}-n_k \rangle\right|&\leq\sum_{|l-l'|\leq 2}\left|\int_\mathbb{T}\Delta_{l}(u_{k+j}-\partial_x u_{k+j}) \cdot\Delta_{l'}(n_{k+j}-n_k)dx\right|\\
 & \leq 2^{3}\sum_{|l-l'|\leq 2}\left( 2^{\frac{3l}{2}}\|\Delta_{l}(u_{k+j}-\partial_x u_{k+j})\|_{L^2} \cdot 2^{-\frac{3l'}{2}}\|\Delta_{l'}(n_{k+j}-n_k)\|_{L^2}\right)\\
&\leq 2^{3}\sup_{l'\geq -1}2^{-\frac{3l'}{2}}\|\Delta_{l'}(n_{k+j}-n_k)\|_{L^2} \cdot\sum_{l\geq -1} 2^{\frac{3l}{2}}\|\Delta_{l}(u_{k+j}-\partial_x u_{k+j})\|_{L^2}\\
&= 2^{3}\|u_{k+j}-\partial_x u_{k+j}\|_{B_{2,1}^{3/2}}\|n_{k+j}-n_k\|_{B_{2,\infty}^{-3/2}}\\
&\leq C\Big(\|u_{k+j}\|_{B_{2,1}^{3/2}}+\| u_{k+j}\|_{B_{2,1}^{5/2}}\Big)\|n_{k+j}-n_k\|_{B_{2,1}^{-1/2}}\\
&\leq C \hbar(F_0)\|n_k-  n_{k+j}\|_{B^{-1/2}_{2,1}},
\end{split}
\end{eqnarray}
where the last two inequality used \eqref{3.3} and the uniform bound for approximate solutions. For the second integral on the right hand side of \eqref{3.25}, we have
\begin{equation}\label{3.27}
\begin{split}
&\left|\langle u_{k+j}-u_k-\partial_x (u_{k+j} - u_k),n_k \rangle\right| \\
&\quad = \left|\sum_{|l-l'|\leq 2}\int_\mathbb{T} \Delta_{l}\left(u_{k+j}-u_k-\partial_x (u_{k+j} - u_k)\right)\cdot \Delta_{l'}n_k dx\right|\\
 &\quad\leq 2 \sum_{|l-l'|\leq 2}2^{\frac{l}{2}}\left\|\Delta_{l}\left(u_{k+j}-u_k-\partial_x (u_{k+j} - u_k)\right)\right\|_{L^2}2^{-\frac{l'}{2}}\cdot\|\Delta_{l'}n_k\|_{L^2}\\
&\quad\leq C\Big(\|u_{k+j}-u_k\|_{B_{2,1}^{1/2}}+\| \partial_x (u_{k+j} - u_k)\|_{B_{2,1}^{1/2}}\Big)\|n_{k+j}-n_k\|_{B_{2,\infty}^{-1/2}}\\
&\quad\leq C\Big(\|n_{k+j}\|_{B_{2,1}^{1/2}}+\|n_k\|_{B_{2,1}^{1/2}}\Big) \Big(\|m_{k+j}-m_k\|_{B_{2,1}^{-3/2}}+\| m_{k+j}-m_k\|_{B_{2,1}^{-1/2}}\Big)\\
&\quad \leq C\hbar(F_0)\|m_k-  m_{k+j}\|_{B^{-1/2}_{2,1}}.
\end{split}
\end{equation}
Inserting the estimates \eqref{3.26} and \eqref{3.27} into \eqref{3.25}, we get
\begin{equation}\label{3.28}
\begin{split}
\left\|\alpha m_{k+j}\overline{\phi_{k,j}}\right\|_{B^{-1/2}_{2,\infty}} \leq C\hbar^2(F_0)|\alpha(t)| \Big(\|m_k-m_{k+j}\|_{B^{-1/2}_{2,1}} +\|n_k-  n_{k+j}\|_{B^{-1/2}_{2,1}}\Big).
\end{split}
\end{equation}
For the fifth term in $\mathcal {F}_1$, we have
\begin{equation}\label{3.29}
\begin{split}
\left\|(\alpha+\gamma)m_{k+j}\overline{\varphi_{k,j}}\right\|_{B^{-1/2}_{2,\infty}} \leq C\hbar^2(F_0)(|\alpha(t)|+|\gamma(t)|) \Big(\|m_k-m_{k+j}\|_{B^{-1/2}_{2,1}} +\|n_k-  n_{k+j}\|_{B^{-1/2}_{2,1}}\Big).
\end{split}
\end{equation}
For the last term in $\mathcal {F}_1$, by choosing
$$
p=2,\quad r=1,\quad s_1=- \frac{1}{2}  \quad\mbox{and}\quad s_2=\frac{1}{2}+\epsilon , ~~\forall \epsilon>0,
$$
it then follows from Lemma \ref{moser}  that
\begin{eqnarray}\label{3.30}
\begin{split}
\|(\rho_{k+j}-\rho_k) \partial_x m_{k+1}\|_{B^{-1/2}_{2,\infty}}
&\leq C\| m_{k+1}\|_{ B^{1/2}_{2,1} }\|(\alpha+\gamma) \varphi_{k,j}(t,x)+\alpha\phi_{k,j}(t,x)\|_{B^{-1/2}_{2,1}}\\
 &\leq C\hbar(F_0)(|\alpha(t)|+|\gamma(t)|) \Big(\|v_k+\partial_x v_k\|_{B^{1/2+\epsilon}_{2,1}}\|m_{k+j}-m_k\|_{B^{-1/2}_{2,1}} \\
&\quad+\|v_k-v_{k+j}+\partial_x (v_k-  v_{k+j})\|_{B^{1/2}_{2,1}} \|m_{k+j} \|_{B^{1/2}_{2,1}} \\
&\quad+ \|u_{k+j}-\partial_x u_{k+j}\|_{B^{1/2+\epsilon}_{2,1}}\|n_{k+j}-n_k\|_{B^{-1/2}_{2,1}}\\
&\quad+ \|u_{k+j}-u_k-\partial_x (u_{k+j} - u_k)\|_{B^{1/2}_{2,1}}\|n_k \|_{B^{1/2}_{2,1}} \Big)\\
&\leq C\hbar^2(F_0)(|\alpha(t)|+|\gamma(t)|) \Big(\|m_k-m_{k+j}\|_{B^{-1/2}_{2,1}} +\|n_k-  n_{k+j}\|_{B^{-1/2}_{2,1}}\Big).
\end{split}
\end{eqnarray}
Combining the estimates \eqref{3.20}, \eqref{3.22}-\eqref{3.24}, \eqref{3.28}-\eqref{3.30}, we obtain
\begin{equation*}
\begin{split}
 &\|\mathcal {F}_1(u_k,u_{j+k},v_j,v_{j+k})\|_{B^{-1/2}_{2,\infty}}\leq C\hbar^2(F_0)(|\alpha(t)|+|\gamma(t)|) \Big(\|m_k-m_{k+j}\|_{B^{-1/2}_{2,1}} +\|n_k-  n_{k+j}\|_{B^{-1/2}_{2,1}}\Big),
\end{split}
\end{equation*}
which together with \eqref{3.19} yield that
\begin{eqnarray}\label{3.31}
\begin{split}
\quad \quad &\|(m_{k+j+1}-m_{k+1})(t)\|_{B^{-1/2}_{2,\infty}}\\
 &\quad\leq C F_02^{-k}+C\hbar^2(F_0)\int_0^t(|\alpha(t')|+|\gamma(t')|)\|(m_{k+j+1}-m_{k+1})(t')\|_{B^{-1/2}_{2,\infty}}dt'\\
 &\quad \quad +C\hbar^2(F_0)\int_0^t(|\alpha(t')|+|\gamma(t')|) \Big(\|(m_{k+j}-m_{k})(t')\|_{B^{-1/2}_{2,1}} +\|(n_{k+j}-n_{k})(t')\|_{B^{-1/2}_{2,1}}\Big)dt'.
\end{split}
\end{eqnarray}
Now we need the following lemma, which provides a characterization for the difference between the Besov spaces $B_{2,1}^s$ and $B_{2,\infty}^s$.
\begin{lemma}[Log-type interpolation inequality \cite{danchin2005fourier}] \label{logestimate}
For any $s\in \mathbb{R}$, $\delta>0$ and $1\leq p \leq \infty$, we have for some constant $C>0$
\begin{equation*}
\begin{split}
 \|f\|_{B^{s}_{p,1}}\leq C\frac{1+\delta}{\delta} \|f\|_{B^{s}_{p,\infty}}\left(1+\log \frac{\|f\|_{B^{s+\delta}_{p,\infty}}}{\|f\|_{B^{s}_{p,\infty}}} \right).
\end{split}
\end{equation*}
\end{lemma}
Applying Lemma \ref{logestimate} with $p=2$, $s=-1/2$ to the right hand side of \eqref{3.31}, and using the following uniform bound (via \eqref{3.14})
$$
\|m_{k+j}-m_{k}\|_{L^\infty_{T^*}(B^{-1/2+\delta}_{2,\infty})}+\|n_{k+j}-n_{k}\|_{L^\infty_{T^*}(B^{-1/2+\delta}_{2,\infty})}\leq 4C\hbar(F_0),\quad \forall k,j\geq 1,
$$
for any $\delta\in (0,1)$, we obtain
\begin{equation}\label{3.32}
\begin{split}
&\|(m_{k+j+1}-m_{k+1})(t)\|_{B^{-1/2}_{2,\infty}}\\
 &\quad\leq C F_02^{-k}+C\hbar^2(F_0)\int_0^t(|\alpha(t')|+|\gamma(t')|)\|(m_{k+j+1}-m_{k+1})(t')\|_{B^{-1/2}_{2,\infty}}dt'\\
 &\qquad +C\hbar^2(F_0)\int_0^t(|\alpha(t')|+|\gamma(t')|) \|(m_{k+j}-m_{k})(t')\|_{B^{-1/2}_{2,\infty}}\\
 &\quad\quad\quad\quad\qquad\times \left(1
 +\log\frac{\|(m_{k+j}-m_{k})(t')\|_{B^{-1/2+\delta}_{2,\infty}}}{\|(m_{k+j}-m_{k})(t')\|_{B^{-1/2}_{2,\infty}}}\right) dt'\\
 &\qquad +C\hbar^2(F_0)\int_0^t(|\alpha(t')|+|\gamma(t')|) \|(n_{k+j}-n_{k})(t')\|_{B^{-1/2}_{2,\infty}}\\
 &\quad\quad\quad\quad\qquad\times \left(1
 +\log\frac{\|(n_{k+j}-n_{k})(t')\|_{B^{-1/2+\delta}_{2,\infty}}}{\|(n_{k+j}-n_{k})(t')\|_{B^{-1/2}_{2,\infty}}}\right) dt'\\
 &\quad\leq C F_02^{-k}+C\hbar^2(F_0)\int_0^t(|\alpha(t')|+|\gamma(t')|)\|(m_{k+j+1}-m_{k+1})(t')\|_{B^{-1/2}_{2,\infty}}dt'\\
 &\qquad+C\hbar^2(F_0)\int_0^t(|\alpha(t')|+|\gamma(t')|) \mathscr{D}_{k,j}(t') \left(1
 +\log\frac{4C\hbar(F_0)}{\mathscr{D}_{k,j}(t')}\right) dt'.
\end{split}
\end{equation}
In the last inequality, we have used the fact that the function $\nu(x)=x(1+\log(4C\hbar(F_0)/x))$ is continuous and increasing for all $x\in (0,4C\hbar(F_0)]$. Similarly, by applying Lemma \ref{lem:priorieti} and the techniques as we used above, one can also derive the estimate for the equation with respect to $n_{k+j+1}-n_{k+1}$ as follows:
\begin{equation}\label{3.33}
\begin{split}
&\|(n_{k+j+1}-n_{k+1})(t)\|_{B^{-1/2}_{2,\infty}}\\
   &\quad\leq C F_02^{-k}+C\hbar^2(F_0)\int_0^t(|\alpha(t')|+|\gamma(t')|)\|(n_{k+j+1}-n_{k+1})(t')\|_{B^{-1/2}_{2,\infty}}dt'\\
 &\qquad+C\hbar^2(F_0)\int_0^t(|\alpha(t')|+|\gamma(t')|) \mathscr{D}_{k,j}(t') \left(1
 +\log\frac{4C\hbar(F_0)}{\mathscr{D}_{k,j}(t')}\right) dt'.
\end{split}
\end{equation}
Adding \eqref{3.32} and \eqref{3.33} leads to
\begin{equation}\label{3.34}
\begin{split}
\mathscr{D}_{k+1,j}(t)&\leq C F_02^{-k}+C\hbar^2(F_0)\int_0^t(|\alpha(t')|+|\gamma(t')|)\mathscr{D}_{k+1,j}(t')dt'\\
 &\quad +C\hbar^2(F_0)\int_0^t(|\alpha(t')|+|\gamma(t')|) \mathscr{D}_{k,j}(t') \left(1
 +\log\frac{4C\hbar(F_0)}{\mathscr{D}_{k,j}(t')}\right) dt'.
\end{split}
\end{equation}
Thanks to the fact of $\alpha,\gamma \in L^1_{loc}([0,\infty);\mathbb{R})$, we have
\begin{equation*}
\begin{split}
\int_{t'}^t(|\alpha(\tau)|+|\gamma(\tau)|)d\tau\leq \int_{0}^{T^*}(|\alpha(\tau)|+|\gamma(\tau)|)d\tau \leq C, \quad  0\leq t'\leq t \leq T^*.
\end{split}
\end{equation*}
An application of Gronwall's lemma to \eqref{3.34} leads to
\begin{equation*}
\begin{split}
\mathscr{D}_{k+1,j}(t)\leq& C F_02^{-k} \exp{\left\{C F_0^2\int_0^t(|\alpha(t')|+|\gamma(t')|)dt'\right\}}\\
&+C\hbar^2(F_0)\int_0^t \exp{\left\{C\hbar^2(F_0)\int_{t'}^t(|\alpha(\tau)|+|\gamma(\tau)|)d\tau\right\}}  (|\alpha(t')|+|\gamma(t')|)\\
&\quad \quad\quad\times \mathscr{D}_{k,j}(t') \left(1
 +\log \frac{4C\hbar(F_0)}{\mathscr{D}_{k,j}(t')} \right) dt'\\
 \leq& e^{C\hbar^2(F_0)}\left(2^{-k} +\int_0^t    ( |\alpha(t')|+|\gamma(t')|)\mathscr{D}_{k,j}(t') \left(1
 +\log\frac{4C\hbar(F_0)}{\mathscr{D}_{k,j}(t')}\right) dt'\right).
\end{split}
\end{equation*}
This proves the inequality \eqref{3.15}.

Now we show the strong convergence result of the approximate solutions in $C_{T^*}(X_{-1/2,\infty})$. Since $\nu(x)$ is an increasing continuous function,  by taking the supremum with respect to $j\geq 1$ in \eqref{3.15}, we obtain
\begin{equation}\label{3.35}
\begin{split}
\sup_{j\geq 1}\mathscr{D}_{k+1,j}(t) \leq  e^{C\hbar^2(F_0)}\left(2^{-k} +\int_0^t     (|\alpha(t')|+|\gamma(t')|)~\sup_{j\geq 1}\mathscr{D}_{k,j}(t') \left(1
 +\log\frac{4C\hbar(F_0)}{\sup\limits_{j\geq 1}\mathscr{D}_{k,j}(t')}\right) dt'\right).
\end{split}
\end{equation}
Define
\begin{equation*}
\begin{split}
\mathscr{E}(t)=\limsup_{k\rightarrow \infty}\sup_{j\geq 1}\mathscr{D}_{k,j}(t).
\end{split}
\end{equation*}
Then for any $\epsilon>0$, there is an integer $n_\epsilon>0$ such that
\begin{equation*}
\begin{split}
\sup_{j\geq 1}\mathscr{D}_{k,j}(t)\leq \mathscr{E}(t)+\epsilon,\quad \forall k\geq n_\epsilon.
\end{split}
\end{equation*}
It then follows from \eqref{3.35} that
\begin{equation}\label{3.36}
\begin{split}
\sup_{j\geq 1}\mathscr{D}_{k+1,j}(t) \leq  e^{C\hbar^2(F_0)}\left(2^{-k} +\int_0^t    ( |\alpha(t')|+|\gamma(t')|)(\mathscr{E}(t')+\epsilon) \left(1
 +\log\frac{4C\hbar(F_0)}{\mathscr{E}(t')+\epsilon}\right) dt'\right).
\end{split}
\end{equation}
Using the fact of
\begin{equation*}
\begin{split}
1+\log (x) =\log(ex)\leq (e+1)\log(e+x),\quad \forall x\geq1,
\end{split}
\end{equation*}
and taking the limit as $k\rightarrow+ \infty$ in \eqref{3.36}, one can derive that
\begin{equation*}
\begin{split}
\mathscr{E}(t)&\leq e^{C\hbar^2(F_0)}\int_0^t     (|\alpha(t')|+|\gamma(t')|)(\mathscr{E}(t')+\epsilon) \left(1
 +\log\frac{4C\hbar(F_0)}{\mathscr{E}(t')+\epsilon}\right) dt'\\
&\leq(e+1)e^{C\hbar^2(F_0)}\int_0^t (|\alpha(t')|+|\gamma(t')|)(\mathscr{E}(t')+\epsilon) \log\left(e
 +\frac{4C\hbar(F_0)}{\mathscr{E}(t')+\epsilon}\right)dt'.
\end{split}
\end{equation*}
Let $\epsilon\rightarrow 0$, we get for any $t\in [0,T^*]$ that
\begin{equation}\label{3.37}
\begin{split}
\mathscr{E}(t)\leq  e^{C\hbar^2(F_0)}\int_0^t   (|\alpha(t')|+|\gamma(t')|)\mu(\mathscr{E}(t')) dt' ,
\end{split}
\end{equation}
where
\begin{equation*}
\begin{split}
\mu(x)=
\begin{cases}
x\log\left(e+\frac{4C\hbar(F_0)}{x}\right),& x\in (0,4C\hbar(F_0)],\\
0,& x=0.
\end{cases}
\end{split}
\end{equation*}
Notice that $\mu(x)$ is an increasing continuous function for $x\in [0,4C\hbar(F_0)]$, and hence a modulus of continuity. Moreover, by the substitution of variable $y=4C\hbar(F_0)/x$, we have
\begin{equation*}
\begin{split}
\int_0^{4C\hbar(F_0)}\frac{dx}{\mu(x)}&=\int_1^{+\infty}\frac{dx}{y\log(e+y)}\\
&\geq \int_1^{+\infty}\frac{dx}{(e+y)\log(e+y)}=\log\log(e+y)|_{1}^{+\infty}=+\infty,
\end{split}
\end{equation*}
which implies that the function $\mu(x)$ is actually an Osgood modulus of continuity on $[0,4C\hbar(F_0)]$.

As both $\alpha$ and $\gamma$ are locally integrable functions from $[0,T^*]$ into $\mathbb{R}^+$, one can now apply the well-known Osgood lemma (cf. Lemma 3.4 in \cite{64}) to \eqref{3.37} to obtain
\begin{equation}\label{3.38}
\begin{split}
\mathscr{E}(t)\equiv 0, \quad \mbox{for all}~ t\in [0,T^*].
\end{split}
\end{equation}
Recalling the definition of $\mathscr{E}(t)$, we obtain from \eqref{3.38} that $\lim_{k\rightarrow \infty}\sup_{j\geq 1}\mathscr{D}_{k+1,j}(t)=0$, which implies that $(m_k,n_k)_{k\geq1}$ is a Cauchy sequence in $C_{T^*}(X_{-1/2,\infty})$.

To improve the convergence to a more regular space $C_{T^*}(X_{-1/2,1})$, we shall apply an interpolation argument, which is based on the following lemma.
\begin{lemma}[Real interpolation inequality \cite{chemin2004localization}] \label{realintepolation}
If $s_1$ and $s_2$ are real numbers such that $s_1<s_2$, $\theta \in (0,1)$, and $(p,r)$  is in $[1,\infty]$, then we have
\begin{equation*}
\begin{split}
 \|f\|_{B^{\theta s_1+ (1-\theta)s_2}_{p,1}}\leq  \frac{C}{s_1-s_2}\left(\frac{1}{\theta}+\frac{1}{1-\theta}\right)
 \|f\|_{B^{s_1}_{p,\infty}}^\theta \|f\|_{B^{s_2}_{p,\infty}}^{1-\theta},
\end{split}
\end{equation*}
for some positive constant $C$.
\end{lemma}
For any $\theta\in (0,1)$, since $1/2-\theta> -1/2$, the Sobolev embedding from $B_{2,1}^{1/2-\theta}$ into  $B_{2,1}^{-1/2}$ is continuous. By choosing $s_1=-1/2$, $s_2=1/2$ and $p=2$ in Lemma \ref{realintepolation}, we have
\begin{eqnarray}\label{3.39}
\begin{split}
\|(m_{k+j+1}-m_{k+1})(t)\|_{B_{2,1}^{-1/2} } &\leq C\|(m_{k+j+1}-m_{k+1})(t)\|_{B_{2,1}^{1/2-\theta} } \\ &\leq C(\theta)\|(m_{k+j+1}-m_{k+1})(t)\|_{B_{2,\infty}^{-1/2} } ^\theta \|(m_{k+j+1}-m_{k+1})(t)\|_{B_{2,\infty}^{1/2} }^{1-\theta}\\
&\leq C(\theta)\|(m_{k+j+1}-m_{k+1})(t)\|_{B_{2,\infty}^{-1/2} } ^\theta(\|m_{k+j+1}\|_{B_{2,1}^{1/2} }+\|m_{k+1}\|_{B_{2,1}^{1/2} })^{1-\theta}\\
&\leq C(\theta)(4C\hbar(F_0))^{1-\theta}\|(m_{k+j+1}-m_{k+1})(t)\|_{B_{2,\infty}^{-1/2} } ^\theta,
\end{split}
\end{eqnarray}
where $2C\hbar(F_0)$ is the uniform bound for the approximate solutions $(m_k,n_k)$ for all $t\in [0,T^*]$. It then  from \eqref{3.39} and \eqref{3.38} that the sequence $(m_k,n_k)_{k\geq1}$ is a Cauchy sequence in $C_{T^*}(X_{1/2,1})$. As a result, there must be a pair of function $(m,n)$ such that
\begin{equation}\label{3.40}
\begin{split}
(m_k,n_k)\rightarrow (m,n)\quad \mbox{strongly in} \quad C_{T^*}(X_{-1/2,1}).
\end{split}
\end{equation}

\subsection{Proof of existence} We verify that the pair $(m,n)$ in \eqref{3.40} is actually the strong solution to the system \eqref{P2NSQQ}. To this end, let us recall the following crucial lemma.
\begin{lemma}[Fatou-type lemma \cite{danchin2005fourier}]
Let $s\in\mathbb{R}$, $(p,r)\in[1,\infty]^2$. If $(f_k)_{k\geq1}$ is a bounded sequence in $B_{p,r}^s$, and $f_k\stackrel{\mathcal {S}'}{\longrightarrow} f$ as $k\rightarrow\infty$, where $\mathcal {S}'$ is the tempered distribution space. Then $u\in B_{p,r}^s$ and
\begin{equation*}
\begin{split}
 \|f\|_{B^{s}_{p,r}}\leq C\liminf_{k\rightarrow\infty}\|f_k\|_{B^{s}_{p,r}}.
\end{split}
\end{equation*}
\end{lemma}

\begin{proof}[Proof of Theorem \ref{theorem1}: Existence] We finish the proof of the existence of local and global solutions to the 2NSQQ system respectively.

(1) \textbf{Local-in-time existence.}
As $B_{p,r}^s\hookrightarrow \mathcal {S}'$ for any $s\in \mathbb{R}$, we  conclude from \eqref{3.40} that $(m_k,n_k)\stackrel{\mathcal {S}'}{\rightarrow} (m,n)$ as $k\rightarrow\infty$. Moreover, by using the fact that $(m_k,n_k)_{k\geq1}$ is uniformly bounded in $C_{T^*}(X_{1/2,1})$, it follows from Lemma 3.5 that $(m,n)\in C_{T^*}(X_{1/2,1})$. In terms of the system \eqref{P2NSQQ} itself, it is easy to verify that $(\partial_tm,\partial_tn)\in C_{T^*}(X_{-1/2,1})$.  Thanks to the strong convergence result \eqref{3.40}, it is then easy to pass to the limit in \eqref{3.1} and to demonstrate that $(m,n)\in E_{2,1}^{1/2}(T^*)$ is indeed a strong solution of the 2NSQQ system.

(2) \textbf{Global-in-time existence.} Based on the local existence theory in (1), it suffices to prove the uniform boundness of the approximate solutions $(m_k,n_k)$ for all  $t\in [0,\infty)$. Indeed, since the time-dependent parameters $\alpha,\gamma \in L^1(0,\infty;\mathbb{R})$, we obtain by \eqref{3.9} that
\begin{equation}\label{3.41}
\begin{split}
\sup_{t\in [0,\infty)}F_{1}(t)
\leq&C\exp{\left\{C\int_0^\infty (|\alpha(t')|+|\gamma(t')|)\left(2C F_0\right)^2 dt'\right\}}\\
&\quad\times\left(F_0+\int_0^\infty (|\alpha(t')|+|\gamma(t')|)\left(2C F_0\right)^3 dt'\right)\\
\leq&2C\left(F_0+8 C^3 F_0^3 A(0,\infty)\right)\exp{\left\{4 C^3 F_0^2A(0,\infty)\right\}}\\
=&  2C\widetilde{\hbar}(F_0),
\end{split}
\end{equation}
where the function $\widetilde{\hbar}(x)$ is defined by
\begin{equation*}
\begin{split}
\widetilde{\hbar}(x)\doteq\left(x+8 C^3 A(0,\infty)x^3\right)\exp{\left\{4 C^3A(0,\infty)x^2\right\}}.
\end{split}
\end{equation*}
Clearly, $\widetilde{\hbar}(x)\geq x$  and the function $\widetilde{\hbar}(x)$ is also a modulus of continuity defined on $[0,\infty)$.

For $F_2(t)$, we deduce from \eqref{3.9} that
\begin{equation}\label{3.42}
\begin{split}
\sup_{t\in [0,\infty)}F_{2}(t)\leq& C\left(F_0+ 8C^3\widetilde{\hbar}^3(F_0)A(0,\infty)\right)\exp{\left\{4C^3\widetilde{\hbar}^2(F_0)A(0,\infty)\right\}}\\
\leq& C \widetilde{\hbar}(F_0)\left(1+ 8C^3\widetilde{\hbar}^2(F_0)A(0,\infty)\right)\exp{\left\{4C^3\widetilde{\hbar}^2(F_0)A(0,\infty)\right\}}\\
\leq&C\widetilde{\hbar}(F_0) \exp{\left\{12C^3\widetilde{\hbar}^2(F_0)A(0,\infty)\right\}}.
\end{split}
\end{equation}
To estimate the right hand side  of \eqref{3.42}, one can assume that the infinite integral $A(0,\infty)$ is sufficiently small such that
\begin{equation}\label{upperbound}
\begin{split}
 A(0,\infty)\leq& \frac{\ln2}{12C^3F_0^2\left(1+8C^3A(0,\infty)F_0^2\right)^2 \exp\{8C^3A(0,\infty)F_0^2\}}\\
 =&\frac{\ln2}{12C^3 \widetilde{\hbar}^2(F_0)}.
\end{split}
\end{equation}
Indeed, one of the sufficient assumptions is given by
$$
 A(0,\infty)\leq \frac{\ln 2 }{24C^3}\Big(\|m_0\|_{B^{1/2}_{2,1}}+\|n_0\|_{B^{1/2}_{2,1}}\Big)^{-2}  .
$$
Then we obtain from \eqref{3.42} and \eqref{upperbound} that
\begin{equation*}
\begin{split}
\sup_{t\in [0,\infty)}F_{2}(t)\leq C\widetilde{\hbar}(F_0) \exp{\left\{12C^3\widetilde{\hbar}^2(F_0)A(0,\infty)\right\}}\leq 2C\widetilde{\hbar}(F_0).
\end{split}
\end{equation*}
Inductively, for any given $n\geq 3$, we assume that
\begin{equation*}
\begin{split}
\sup_{t\in [0,\infty)}F_{n}(t)\leq 2C\widetilde{\hbar}(F_0).
\end{split}
\end{equation*}
Using \eqref{3.9} and the upper bound \eqref{upperbound} again, we have
\begin{equation*}
\begin{split}
\sup_{t\in [0,\infty)}F_{n+1}(t)
\leq &C\exp{\left\{4C^3\widetilde{\hbar}^2(F_0)A(0,\infty)\right\}}
\left(F_0+8C^3\widetilde{\hbar}^3(F_0)A(0,\infty)\right)\\
\leq &C\widetilde{\hbar}(F_0)\exp{\left\{4C^3\widetilde{\hbar}^2(F_0)A(0,\infty)\right\}}
\left(1+8C^3\widetilde{\hbar}^2(F_0)A(0,\infty)\right)\\
\leq &C\widetilde{\hbar}(F_0)\exp{\left\{12C^3\widetilde{\hbar}^2(F_0)A(0,\infty)\right\}} \\
\leq &2C\widetilde{\hbar}(F_0).
\end{split}
\end{equation*}
Therefore, we have proved the following uniform bound for approximate solutions
\begin{equation}\label{3.43}
\begin{split}
 \sup_{t\in [0,\infty)}\Big(\|m_k(t)\|_{B^{1/2}_{2,1}}+\|n_k(t)\|_{B^{1/2}_{2,1}}\Big)\leq 2C\widetilde{\hbar}(F_0),\quad \mbox{for any}~ k \geq 0,
\end{split}
\end{equation}
which implies that the sequence $(m_k,n_k)_{k\geq 1}$ is uniformly bounded with respect to $t$-variable on $[0,\infty)$, and hence ensures the existence of global-in-time solution to the 2NSQQ system \eqref{P2NSQQ} along with the conclusion in (1).

The proof of the existence part of Theorem \ref{theorem1} is now completed.
\end{proof}

\section{Proof of Theorem \ref{theorem1}: Uniqueness}
In this section, we aim at establishing the uniqueness of the solution $(m,n)$ to the periodic Cauchy problem \eqref{P2NSQQ}, which is a direct consequence of the following lemma.
\begin{lemma} \label{uniqueness}
Let $(m_j,n_j)\in X_{T^*}$ be a solution to the periodic Cauchy problem \eqref{P2NSQQ} with  initial data $(m_j(x,0),n_j(x,0))\in X_{1/2,1}$ ($j=1,2$). Setting $\varpi(x,t)=m_1(x,t)-m_2(x,t)$, $\omega(x,t)=n_1(x,t)-n_2(x,t)$, $\varpi_0(x)=m_1(x,0)-m_2(x,0)$, and $\omega_0(x)=n_1(x,0)-n_2(x,0)$. If we denote
$$
\mathscr{D} (t)= \sum_{j=1,2}\Big(\|m_j(t)\|_{B^{1/2}_{2,1}}^2+\|n_j(t)\|_{B^{1/2}_{2,1}}^2\Big),\quad \mathscr{D}_0= \sum_{j=1,2}\Big(\|m_j(0)\|_{B^{1/2}_{2,1}}^2+\|n_j(0)\|_{B^{1/2}_{2,1}}^2\Big),
$$
then we have
\begin{equation*}
\begin{split}
 &\frac{\|\varpi(t)\|_{B^{-1/2}_{2,\infty}}+\|\omega(t)\|_{B^{-1/2}_{2,\infty}}}{4 eC \hbar\left(4+\mathscr{D} _0\right)}\\
 &\quad\leq  \exp{\left\{16 C^2 \hbar^2\left(4+\mathscr{D} _0\right)\int_0^t(|\alpha(t')|+|\gamma(t')| ) dt' \right\}} \left(\frac{\|\varpi_0\|_{B^{-1/2}_{2,\infty}}+\|\omega_0\|_{B^{-1/2}_{2,\infty}}}{4 eC \hbar\left(4+\mathscr{D} _0\right)}\right)^{\sigma(t)},
\end{split}
\end{equation*}
for all $t\in[0,T^*]$, where
\begin{equation*}
\begin{split}
\sigma(t)=\exp{\left\{-16 C^2\log(e+1) \hbar^2\left(4+\mathscr{D} _0\right)\int_{0}^t (|\alpha(t')|+|\gamma(t')| )dt'\right\}},
\end{split}
\end{equation*}
and the the function $\hbar(x)$ is a modulus of continuity defined in \eqref{3.10}.
\end{lemma}

\begin{proof} [Proof of Lemma \ref{uniqueness}]
Apparently,  the pair of functions $(\varpi,\omega)$ solves the periodic Cauchy problem of the following nonlinear transport-type system:
\begin{equation}\label{4.1}
\begin{cases}
\varpi_t+\rho_1 \partial_x \varpi
=-(\rho_1 - \rho_2)\partial_xm_2-\varpi\left(\psi_2(t,x)-\overline{\psi_2}(t)\right)\\
\quad \quad\quad\quad\quad\quad+(\alpha+\gamma)m_1 \varphi(t,x)-\alpha m_1\phi(t,x)-(\alpha+\gamma)m_1\overline{\varphi}(t)+\alpha m_1\overline{\phi}(t),\\

\omega_t+\rho_1 \partial_x \omega
=-(\rho_1 - \rho_2)\partial_xn_2-\omega\left(\psi_2(t,x)-\overline{\psi_2}(t)\right)\\
\quad \quad\quad\quad\quad\quad+(\alpha+\gamma)n_1 \varphi(t,x)-\alpha n_1\phi(t,x)-(\alpha+\gamma)n_1\overline{\varphi}(t)+\alpha n_1\overline{\phi}(t),\\

\varpi(x,0)=\varpi_0(x),\quad \omega(x,0)=\omega_0(x),
\end{cases}
\end{equation}
where $m_j=u_j-\partial_x^2 u_j$, $n_j=v_j-\partial_x^2 v_j$, and
\begin{equation*}
\begin{split}
\varphi(t,x)&= [v_1-v_2+\partial_x(v_1-v_2)]m_1+(v_2+\partial_x v_2)\varpi,\\
\phi(t,x)&=[u_1-u_2-\partial_x(u_1-u_2)]n_1+(u_2-\partial_x u_2)\omega.
\end{split}
\end{equation*}
By applying Lemma   \ref{lem:priorieti} to the first equation with respect to $\varpi$ in \eqref{4.1}, we have
\begin{equation}\label{4.2}
\begin{split}
 \|\varpi(t)\|_{B^{-1/2}_{2,\infty}}\leq & \|\varpi_0\|_{B^{-1/2}_{2,\infty}}+C\int_0^t\|\partial_x\rho_1(t')\|_{B^{1/2}_{2,\infty}\cap L^\infty}\|\varpi(t')\|_{B^{-1/2}_{2,\infty}}dt' +\sum_{j=1}^4\int_0^tT_j(t') dt',
\end{split}
\end{equation}
where
\begin{equation*}
\begin{split}
&T_1(t) =  \|(\rho_1 - \rho_2)\partial_xm_2\|_{B^{-1/2}_{2,\infty}},\\ &T_2(t)=\|\varpi\left(\psi_2(t,\cdot)-\overline{\psi_2}(t)\right)\|_{B^{-1/2}_{2,\infty}},\\
&T_3(t) = \|(\alpha+\gamma)m_1 \varphi(t,\cdot)\|_{B^{-1/2}_{2,\infty}}+\|\alpha m_1\phi(t,\cdot)\|_{B^{-1/2}_{2,\infty}},\\
&T_4(t) =\|(\alpha+\gamma)m_1\overline{\varphi}(t)\|_{B^{-1/2}_{2,\infty}}+\|\alpha m_1\overline{\phi}(t)\|_{B^{-1/2}_{2,\infty}}.
\end{split}
\end{equation*}
Using the embedding $B^{ 1/2}_{2,1}\hookrightarrow B^{1/2}_{2,\infty}\cap L^\infty$ and the algebra property, we first have
\begin{equation}\label{4.3}
\begin{split}
\|\partial_x\rho_1\|_{B^{1/2}_{2,\infty}\cap L^\infty}=&\|\psi_1(t,\cdot)\|_{B^{1/2}_{2,\infty}\cap L^\infty}+\|\psi_1 (t,\cdot)\|_{L^\infty}\\
\leq& \|(\alpha+\gamma)(v_1+\partial_x v_1)m_1-\alpha(u_1-\partial_x u_1)n_1\|_{B^{1/2}_{2,\infty}\cap L^\infty}\\
\leq& |\alpha+\gamma |\|v_1+\partial_x v_1\|_{B^{ 1/2}_{2,1}}\|m_1\|_{B^{ 1/2}_{2,1}}+|\alpha|\|u_1-\partial_x u_1\|_{B^{ 1/2}_{2,1}}\|n_1\|_{B^{ 1/2}_{2,1}}\\
\leq& C(|\alpha|+|\gamma |) \|m_1\|_{B^{ 1/2}_{2,1}} \|n_1\|_{B^{ 1/2}_{2,1}}.
\end{split}
\end{equation}
Let us now estimate the terms $T_j(t)$ ($j=1,...,6$) in the inequality \eqref{4.2} one by one. Notice that
$$
(\psi_1 - \psi_2)(t,x)= (\alpha+\gamma ) \varphi(t,x)+\alpha \phi(t,x).
$$
For $T_1(t)$, it follows from Lemma   \ref{endpoint} that
\begin{equation}\label{4.4}
\begin{split}
T_1(t)  &\leq C \|m_2\|_{B^{1/2}_{2,1}}\Big((|\alpha|+|\gamma| ) \|\varphi(t,\cdot)\|_{B^{-1/2}_{2,1}}+|\alpha| \|\phi(t,\cdot)\|_{B^{-1/2}_{2,1}}\Big)\\
 &\leq C (|\alpha|+|\gamma| ) \|m_2\|_{B^{1/2}_{2,1}}\Big(\|v_1-v_2+\partial_x(v_1-v_2)\|_{B^{ 1/2}_{2,1}}\|m_1\|_{B^{1/2}_{2,1}}\\
 &\quad +\|v_2+\partial_x v_2\|_{B^{3/2}_{2,1}}\|\varpi\|_{B^{-1/2}_{2,1}}+\|u_2-\partial_x u_2\|_{B^{1/2}_{2,1}}\|\omega\|_{B^{-1/2}_{2,1}}\\
&\quad + \|u_1-u_2-\partial_x(u_1-u_2)\|_{B^{ 1/2}_{2,1}}\|n_1\|_{B^{1/2}_{2,1}}\Big)\\
&\leq C (|\alpha|+|\gamma| )\mathscr{D} (t) \|m_2\|_{B^{1/2}_{2,1}}\Big(\|\omega\|_{B^{-1/2}_{2,1}}+\|\varpi\|_{B^{-1/2}_{2,1}}\Big).
\end{split}
\end{equation}
For $T_2(t)$, we deduce by Lemma \ref{endpoint} that
\begin{equation}\label{4.5}
\begin{split}
T_2(t)&\leq\| \psi_2(t,x)-\overline{\psi_2}(t) \|_{B^{1/2}_{2,\infty}\cap L^\infty}\|\varpi\|_{B^{-1/2}_{2,1}}\\
&\leq C\left\| (\alpha+\gamma)(v_2+\partial_x v_2)m_2-\alpha(u_2-\partial_x u_2)n_2\right\|_{B^{1/2}_{2,1} }\|\varpi\|_{B^{-1/2}_{2,1}}\\
&\leq C (|\alpha|+|\gamma| ) \|m_2\|_{B^{1/2}_{2,1}}\|n_2\|_{B^{1/2}_{2,1}}\|\varpi\|_{B^{-1/2}_{2,1}} .
\end{split}
\end{equation}
For $T_3(t)$, we have
\begin{equation}\label{4.6}
\begin{split}
T_3(t) & \leq C (|\alpha|+|\gamma| )\|m_1\|_{B^{1/2}_{2,\infty}\cap L^\infty} \Big( \|\varphi(t,\cdot)\|_{B^{-1/2}_{2,1}}+ \|\phi(t,\cdot)\|_{B^{-1/2}_{2,1}} \Big)\\
&\leq C (|\alpha|+|\gamma| ) \mathscr{D} (t)\|m_1\|_{B^{1/2}_{2,1}}\Big(\|\omega\|_{B^{-1/2}_{2,1}}+\|\varpi\|_{B^{-1/2}_{2,1}}\Big).
\end{split}
\end{equation}
For $T_4(t)$, since the Schwartz space $\mathcal {S}$ is dense in $B_{2,1}^{\pm1/2}$, by using the Littlewood-Paley theory and the density argument similar to \eqref{3.26}, one can derive that
\begin{equation}\label{4.7}
\begin{split}
T_4(t)  \leq& C(|\alpha|+|\gamma| )\|m_1\|_{B^{1/2}_{2,\infty}\cap L^\infty} \left(|\overline{\varphi}(t)|+|\overline{\phi}(t)| \right)\\
\leq & C(|\alpha|+|\gamma| )\|m_1\|_{B^{1/2}_{2,\infty}\cap L^\infty} \Big(\Big| \langle v_1-v_2+\partial_x(v_1-v_2),m_1 \rangle_{B_{2,\infty}^{-1/2},B_{2,1}^{ 1/2}}\Big|\\
&+\Big|\langle v_2+\partial_x v_2 ,\varpi\rangle_{B_{2,\infty}^{1/2},B_{2,1}^{-1/2}}\Big|+\Big|\langle u_2-\partial_x u_2,\omega\rangle_{B_{2,\infty}^{1/2},B_{2,1}^{-1/2}}\Big|\\
&+\Big|\langle u_1-u_2-\partial_x(u_1-u_2),n_1\rangle_{B_{2,\infty}^{1/2},B_{2,1}^{-1/2}}\Big| \Big)\\
\leq & C(|\alpha|+|\gamma| )\|m_1\|_{B^{1/2}_{2,1}} \Big(\| v_1-v_2 \|_{B_{2,\infty}^{ 1/2}}\|m_1 \|_{B_{2,1}^{ 1/2}}+\|v_2 \|_{B_{2,\infty}^{3/2}} \|\varpi\|_{B_{2,1}^{-1/2}}\\
&+\| u_2 \|_{B_{2,\infty}^{3/2}}\|\omega\|_{B_{2,1}^{-1/2}} +\|u_1-u_2 \|_{B_{2,\infty}^{3/2}}\|n_1\|_{B_{2,1}^{-1/2}}\Big)\\
\leq & C(|\alpha|+|\gamma| )\|m_1\|_{B^{1/2}_{2,1}}\Big(\|\omega\|_{B_{2,\infty}^{-3/2}}\|m_1 \|_{B_{2,1}^{ 1/2}}+\|n_2 \|_{B_{2,\infty}^{-1/2}} \|\varpi\|_{B_{2,1}^{-1/2}}\\
&+\| m_2 \|_{B_{2,\infty}^{-1/2}}\|\omega\|_{B_{2,1}^{-1/2}} +\|\varpi\|_{B_{2,\infty}^{-1/2}}\|n_1\|_{B_{2,1}^{-1/2}}\Big)\\
\leq & C(|\alpha|+|\gamma| )\mathscr{D} (t) \Big(\|\omega\|_{B^{-1/2}_{2,1}}+\|\varpi\|_{B^{-1/2}_{2,1}}\Big).
\end{split}
\end{equation}
Inserting the estimates \eqref{4.3}-\eqref{4.7} into \eqref{4.2}, we obtain
\begin{equation*}
\begin{split}
 \|\varpi(t)\|_{B^{-1/2}_{2,\infty}}\leq & \|\varpi_0\|_{B^{-1/2}_{2,\infty}}+C\int_0^t(|\alpha|+|\gamma |) \|m_1 \|_{B^{ 1/2}_{2,1}} \|n_1 \|_{B^{ 1/2}_{2,1}}\|\varpi(t')\|_{B^{-1/2}_{2,\infty}}dt' \\
 &+ C\int_0^t(|\alpha|+|\gamma| ) \mathscr{D} (t')\Big(\|\omega(t')\|_{B^{-1/2}_{2,1}}+\|\varpi(t')\|_{B^{-1/2}_{2,1}}\Big)dt'.
\end{split}
\end{equation*}
Estimating in a similar manner, one can also investigate the second equation in \eqref{4.1} with respect to $\omega$ and derive the following estimate
\begin{equation*}
\begin{split}
 \|\omega(t)\|_{B^{-1/2}_{2,\infty}}\leq & \|\omega_0\|_{B^{-1/2}_{2,\infty}}+C\int_0^t(|\alpha|+|\gamma |) \|m_2 \|_{B^{ 1/2}_{2,1}} \|n_2 \|_{B^{ 1/2}_{2,1}}\|\omega(t')\|_{B^{-1/2}_{2,\infty}}dt' \\
 &+ C\int_0^t(|\alpha|+|\gamma| )\mathscr{D} (t')\Big(\|\varpi(t')\|_{B^{-1/2}_{2,1}}+\|\omega(t')\|_{B^{-1/2}_{2,1}}\Big)dt'.
\end{split}
\end{equation*}
Thereby it follows from the last two estimates that
\begin{equation}\label{4.8}
\begin{split}
 &\|\varpi(t)\|_{B^{-1/2}_{2,\infty}}+\|\omega(t)\|_{B^{-1/2}_{2,\infty}}\\
 &\quad\leq \|\varpi_0\|_{B^{-1/2}_{2,\infty}}+\|\omega_0\|_{B^{-1/2}_{2,\infty}} +C\int_0^t(|\alpha|+|\gamma |) \mathscr{D} (t')\Big(\|\varpi(t')\|_{B^{-1/2}_{2,\infty}}+\|\omega(t')\|_{B^{-1/2}_{2,\infty}}\Big)dt' \\
 &\qquad + C\int_0^t(|\alpha|+|\gamma| )\mathscr{D} (t')\Big(\|\omega(t')\|_{B^{-1/2}_{2,1}}+\|\varpi(t')\|_{B^{-1/2}_{2,1}}\Big)dt'.
\end{split}
\end{equation}
Applying the Gronwall's lemma to the above inequality and  using the logarithmic interpolation inequality (see Lemma \ref{logestimate}), we have
\begin{equation}\label{4.9}
\begin{split}
 &\|\varpi(t)\|_{B^{-1/2}_{2,\infty}}+\|\omega(t)\|_{B^{-1/2}_{2,\infty}}\\
 &\quad\leq   \Big(\|\varpi_0\|_{B^{-1/2}_{2,\infty}}+\|\omega_0\|_{B^{-1/2}_{2,\infty}}\Big)\exp{\left\{\int_0^t(|\alpha|+|\gamma| )\mathscr{D} (t') dt' \right\}}\\
 &\quad\quad + \int_0^t(|\alpha|+|\gamma| )\mathscr{D} (t')\exp{\left\{\int_{t'}^t(|\alpha|+|\gamma| )\mathscr{D} (\tau) d\tau  \right\}}\Big(\|\omega(t')\|_{B^{-1/2}_{2,1}}+\|\varpi(t')\|_{B^{-1/2}_{2,1}}\Big) dt'\\
 & \quad\leq   \Big(\|\varpi_0\|_{B^{-1/2}_{2,\infty}}+\|\omega_0\|_{B^{-1/2}_{2,\infty}}\Big)\exp{\left\{\int_0^t(|\alpha|+|\gamma| )\mathscr{D} (t') dt' \right\}}\\
 &\quad\quad + \int_0^t(|\alpha|+|\gamma| )\mathscr{D} (t') \exp{\left\{\int_{t'}^t(|\alpha|+|\gamma| )\mathscr{D} (\tau) d\tau  \right\}} \Big(\|\varpi(t')\|_{B^{-1/2}_{2,\infty}}+\|\omega(t')\|_{B^{-1/2}_{2,\infty}}\Big)\\
 &\quad\quad \times \log\left(e+ \frac{ \|\varpi\|_{B^{-1/2+\epsilon}_{2,\infty}}+\|\omega\|_{B^{-1/2+\epsilon}_{2,\infty}}}{\|\varpi\|_{B^{-1/2}_{2,\infty}}+\|\omega\|_{B^{-1/2}_{2,\infty}}}  \right) dt'\\
 &\quad \leq  \exp{\left\{\int_0^t(|\alpha|+|\gamma| )\mathscr{D} (t') dt' \right\}}\bigg[ \|\varpi_0\|_{B^{-1/2}_{2,\infty}}+\|\omega_0\|_{B^{-1/2}_{2,\infty}}+ \int_0^t(|\alpha|+|\gamma| )\mathscr{D} (t') \\
 &\quad \quad\times\exp{\left\{-\int_0^{t'}(|\alpha|+|\gamma| )\mathscr{D} (\tau) d\tau  \right\}} \Big(\|\varpi(t')\|_{B^{-1/2}_{2,\infty}}+\|\omega(t')\|_{B^{-1/2}_{2,\infty}}\Big)\\
 &\quad\quad \times \log\left(e+ \frac{\|\varpi\|_{B^{-1/2+\epsilon}_{2,\infty}}+\|\omega\|_{B^{-1/2+\epsilon}_{2,\infty}}}{\exp{\left\{-\int_0^{t'}(|\alpha|+|\gamma| )\mathscr{D} (\tau) d \tau \right\}}\Big(\|\varpi\|_{B^{-1/2}_{2,\infty}}+\|\omega\|_{B^{-1/2}_{2,\infty}}\Big)}  \right) dt'\bigg].
\end{split}
\end{equation}
Setting
$$
S(t)=\exp{\left\{-\int_0^t(|\alpha(t')|+|\gamma(t')| )\mathscr{D} (t') dt' \right\}}\Big(\|\varpi(t)\|_{B^{-1/2}_{2,\infty}}+\|\omega(t)\|_{B^{-1/2}_{2,\infty}}\Big).
$$
Due to the uniform bound \eqref{3.14} and the Fatou-type lemma, we have
\begin{equation*}
\begin{split}
\sum_{j=1,2}\Big(\|m_j(t)\|_{B^{1/2}_{2,1}}+\|n_j(t)\|_{B^{1/2}_{2,1}}\Big)\leq 4 C \hbar\left(\sum_{j=1,2}(\|m_j(0)\|_{B^{1/2}_{2,1}}+\|n_j(0)\|_{B^{1/2}_{2,1}}) \right)\leq 4 C \hbar\left(4+\mathscr{D} _0\right),
\end{split}
\end{equation*}
where the increasing function $\hbar(x)$ is defined in Section 3.

From the definition of $\mathscr{D} (t)$ and $S(t)$ we have
\begin{equation*}
\begin{split}
\mathscr{D} (t)\leq 16 C^2 \hbar^2\left(4+\mathscr{D} _0\right),
\end{split}
\end{equation*}
and
\begin{equation*}
\begin{split}
S(t)&\leq \sup_{t\in [0,T^{*}]}\left[\exp{\left\{-\int_0^t(|\alpha|+|\gamma| )\mathscr{D} (t') dt' \right\}}\Big(\|\varpi(t)\|_{B^{-1/2}_{2,\infty}}+\|\omega(t)\|_{B^{-1/2}_{2,\infty}}\Big)\right]\\
& \leq \sup_{t\in [0,T^{*}]}\Big(\|\varpi(t)\|_{B^{-1/2}_{2,\infty}}+\|\omega(t)\|_{B^{-1/2}_{2,\infty}}\Big)\\
&\leq \sup_{t\in [0,T^{*}]}\Big(\|\varpi(t)\|_{B^{-1/2+\epsilon}_{2,\infty}}+\|\omega(t)\|_{B^{-1/2+\epsilon}_{2,\infty}}\Big)\\
& \leq \sup_{t\in [0,T^{*}]}  \sum_{j=1,2}\Big(\|m_j(t)\|_{B^{1/2}_{2,1}}+\|n_j(t)\|_{B^{1/2}_{2,1}}\Big) \leq 4 C \hbar\left(4+\mathscr{D} _0\right).
\end{split}
\end{equation*}
It then follows from the inequality \eqref{4.9} that
\begin{equation}\label{4.10}
\begin{split}
 S(t)\leq S(0) +16 C^2 \hbar^2\left(4+\mathscr{D} _0\right)\int_0^t (|\alpha(t')|+|\gamma(t')| ) F(S(t')) dt',
\end{split}
\end{equation}
where
\begin{equation*}
F(x)=\begin{cases}
x\log\left(e+ \frac{4 C \hbar\left(4+\mathscr{D} _0\right)}{x}  \right),& x\in (0,4 C \hbar\left(4+\mathscr{D} _0\right)],\\
 0,& x=0.
\end{cases}
\end{equation*}
It is easy to verify that the function  $F(x)$ is an Osgood modulus of continuity on $[0,4 C \hbar\left(4+\mathscr{D} _0\right)]$, so we get by applying the Osgood lemma to inequality \eqref{4.10} that
\begin{equation}\label{4.11}
\begin{split}
 \int_{S(0)}^{S(t)}\frac{dr}{F(r)} &\leq 16 C^2 \hbar^2\left(4+\mathscr{D} _0\right)\int_{0}^t (|\alpha(t')|+|\gamma(t')| )dt'.
\end{split}
\end{equation}
Notice that for any given constant $C>0$, we have the inequality
$$
x\log\left(e+ \frac{C}{x}\right)\leq \log(e+C)(1-\log x), \quad \forall x\in (0,1].
$$
The left hand side of \eqref{4.11} can be estimated as
\begin{equation*}
\begin{split}
\int_{S(0)}^{S(t)}\frac{dr}{F(r)}
 &= \int_{S(0)}^{S(t)}\frac{d r}{r \log\left(e+ \frac{ 1}{ \frac{r}{4 C \hbar\left(4+\mathscr{D} _0\right)}}  \right)}\\
 &\geq \int_{S(0)}^{S(t)}\frac{d(\frac{r}{4 C \hbar\left(4+\mathscr{D} _0\right)})}{\log(e+1)\frac{r}{4 C \hbar\left(4+\mathscr{D} _0\right)}\left(1-\log\frac{r}{4 C \hbar\left(4+\mathscr{D} _0\right)}\right)}\\
 &=-\frac{1}{\log(e+1)}\log\left(\frac{1-\log\frac{S(t) }{4 C \hbar\left(4+\mathscr{D} _0\right)}}{1-\log \frac{S(0)}{4 C \hbar\left(4+\mathscr{D} _0\right)}}\right),
\end{split}
\end{equation*}
which combined with \eqref{4.11} yield that
\begin{equation}\label{4.12}
\begin{split}
 \log\left(\frac{1-\log\frac{S(0) }{4 C \hbar\left(4+\mathscr{D} _0\right)} }{1-\log\frac{S(t)}{4 C \hbar\left(4+\mathscr{D} _0\right)} }\right)\leq16\log(e+1) C^2 \hbar^2\left(4+\mathscr{D} _0\right)\int_{0}^t (|\alpha(t')|+|\gamma(t')| )dt'.
\end{split}
\end{equation}
Solving the inequality \eqref{4.12} leads to
\begin{equation*}
\begin{split}
 \frac{S(t)}{4 eC \hbar\left(4+\mathscr{D} _0\right)}\leq \left(\frac{S(0)}{4 eC \Theta\left(4+\mathscr{D} _0\right)}\right)^{\exp{\{-16 C^2\log(e+1) \hbar^2\left(4+\mathscr{D} _0\right)\int_{0}^t (|\alpha(t')|+|\gamma(t')| )dt'\}}}.
\end{split}
\end{equation*}
From the definition of $S(t)$, the previous inequality implies the desired inequality. This completes the proof of Lemma \ref{uniqueness}.
\end{proof}

\section{Proof of Theorem \ref{theorem1}: Continuity}
In this section, we investigate the continuous property of the data-to-solution map $(m,n)=\Lambda(m_0,n_0)$ for the 2NSQQ system \eqref{P2NSQQ}.

\begin{proof} [Proof of Theorem \ref{theorem1} (Continuity)]
The proof of the continuity of the data-to-solution map will be divided into the following two steps.

\textbf{Step 1:} Continuity of the data-to-solution map in $C_T(X_{-1/2,1})$.  Assume that $(m_0,n_0)\in X_{1/2,1}$, $r>0$ is an arbitrary fixed number, and we define the closed bounded ball $B_r(m_0,n_0)\subseteq X_{1/2,1}$ as
\begin{equation*}
\begin{split}
B_r(m_0,n_0)\doteq\Big\{(\overline{m}_0,\overline{n}_0)\in X_{1/2,1};~~ \|m_0-\overline{m}_0\|_{B_{2,1}^{1/2}}+\|n_0-\overline{n}_0\|_{B_{2,1}^{1/2}}\leq r\Big\}.
\end{split}
\end{equation*}

We make a \textbf{Claim}: there is a $T>0$ and a $M>0$ such that for any $(\overline{m}_0,\overline{n}_0)\in B_r(m_0,n_0)$, the solution $(\overline{m},\overline{n})=\Lambda (\overline{m}_0,\overline{n}_0)$ to the system \eqref{P2NSQQ} belongs to $C_T(X_{1/2,1})$ and satisfies
\begin{equation}\label{5.1}
\begin{split}
\sup_{t\in [0,T]} \Big(\|\overline{m}(t)\|_{B_{2,1}^{1/2}} + \|\overline{n}(t)\|_{B_{2,1}^{1/2}}\Big)\leq M.
\end{split}
\end{equation}

Indeed, it follows from the uniform bound that, for any $(\overline{m}_0,\overline{n}_0)\in \partial B_r(m_0,n_0) $, i.e., $ \|\overline{m}_0\|_{B_{2,1}^{1/2}}+\|\overline{n}_0\|_{B_{2,1}^{1/2}}=\|m_0\|_{B_{2,1}^{1/2}}+\|n_0\|_{B_{2,1}^{1/2}}+r$, one can find a lifespan $T^*$ satisfying
\begin{equation*}
\begin{split}
\int_0^{T^*}(|\alpha(t')|+|\gamma(t')|)dt' \leq \frac{\ln 2}{12C^3\hbar^2\Big(\|m_0\|_{B_{2,1}^{1/2}}+\|n_0\|_{B_{2,1}^{1/2}}+r\Big)}
\end{split}
\end{equation*}
such that the corresponding solution $(\overline{m},\overline{n})$ is bounded by
\begin{equation*}
\begin{split}
 \|\overline{m}\|_{L^\infty_{T^*}(B^{1/2}_{2,1})}+\|\overline{n}\|_{L^\infty_{T^*}(B^{1/2}_{2,1})}\leq 2C\hbar\Big(\|m_0\|_{B_{2,1}^{1/2}}+\|n_0\|_{B_{2,1}^{1/2}}+r\Big),
\end{split}
\end{equation*}
where $\hbar(x)$ is a modulus of continuity defined in (1) of Theorem \ref{theorem1}. Notice that the lifespan $T^*$ is a decreasing function with resect to the norm of the initial data $(m_0,n_0)$. Thereby for any solution $(\overline{m},\overline{n})$ to the system \eqref{P2NSQQ} associated with $(\overline{m}_0,\overline{n}_0)\in B_r(m_0,n_0) $, one can restrict $(\overline{m},\overline{n})$ on the interval $[0,T^*]\subset [0,\overline{T}]$, and
\begin{equation}\label{5.2}
\begin{split}
\sup_{t\in [0,T^*]} \Big(\|\overline{m}(t)\|_{B_{2,1}^{1/2}} + \|\overline{n}(t)\|_{B_{2,1}^{1/2}}\Big)&\leq \sup_{t\in [0,\overline{T}]} \Big(\|\overline{m}(t)\|_{B_{2,1}^{1/2}} + \|\overline{n}(t)\|_{B_{2,1}^{1/2}}\Big)\\
&\leq  2C\hbar\Big(\|\overline{m}_0\|_{B_{2,1}^{1/2}}+\|\overline{n}_0\|_{B_{2,1}^{1/2}}\Big)\\
&\leq 2C\hbar\Big(\|m_0\|_{B_{2,1}^{1/2}}+\|n_0\|_{B_{2,1}^{1/2}}+r\Big)\doteq M.
\end{split}
\end{equation}
Then we prove the Claim by choosing $T=T^*$.

Combining the above conclusion with Lemma   \ref{uniqueness}, we obtain
\begin{equation}\label{5.3}
\begin{split}
&\|\Lambda (m_0,n_0)-\Lambda (\overline{m}_0 ,\overline{n}_0 )\|_{X_{-1/2,\infty}} =\|(m-\overline{m})(t)\|_{B_{2,\infty}^{-1/2}} + \|(n-\overline{n})(t)\|_{B_{2,\infty}^{-1/2}}\\
&\quad \leq  C \hbar\left(4+4M^2\right)\exp{\left\{C\hbar^2\left(4+4M^2\right) \right\}} \\
&\quad\quad \times\left(\frac{\|m_0-m_0'\|_{B^{-1/2}_{2,\infty}}+\|n_0-n_0'\|_{B^{-1/2}_{2,\infty}}}{C \hbar\left(4+4M^2\right)}\right)^{\exp{\left\{-C \hbar\left(4+4M^2\right)^2\right\}}}.
\end{split}
\end{equation}
By using the real interpolation inequality in Lemma \ref{realintepolation}, we deduce from \eqref{5.2}-\eqref{5.3} and the embedding $X_{-1/2+\theta,1}\hookrightarrow X_{-1/2,1}$ for any $\theta\in (0,1)$ that
\begin{equation}\label{5.4}
\begin{split}
&\|\Lambda (m_0,n_0)-\Lambda (\overline{m}_0,\overline{n}_0)\|_{X_{-1/2,1}} \leq \|\Lambda (m_0,n_0)-\Lambda (\overline{m}_0,\overline{n}_0)\|_{X_{-1/2+\theta,1}}\\
&\quad \leq C\|\Lambda (m_0,n_0)-\Lambda (\overline{m}_0,\overline{n}_0)\|_{X_{-1/2,\infty}}^{1-\theta}\|\Lambda (m_0,n_0)-\Lambda (\overline{m}_0,\overline{n}_0)\|_{X_{1/2,\infty}}^\theta\\
&\quad \leq  C \hbar^{1-\theta}\left(4+4M^2\right)\exp{\left\{C(1-\theta)\hbar^2\left(4+4M^2\right) \right\}} \\
&\quad\quad \times\left(\frac{\|m_0-\overline{m}_0 \|_{B^{-1/2}_{2,\infty}}+\|n_0-\overline{n}_0 \|_{B^{-1/2}_{2,\infty}}}{C \hbar\left(4+4M^2\right)}\right)^{\exp{\left\{-C (1-\theta)\hbar\left(4+4M^2\right)^2\right\}}}\\
&\qquad \times \Big(\|m \|_{B_{2,1}^{1/2}}+\|\overline{m} \|_{B_{2,1}^{1/2}}+\|n\|_{B_{2,1}^{1/2}}+\|\overline{n} \|_{B_{2,1}^{1/2}}\Big)^\theta\\
&\quad \leq  C M^\theta\hbar^{1-\theta}\left(4+4M^2\right)\exp{\left\{C(1-\theta)\hbar^2\left(4+4M^2\right) \right\}} \\
&\quad\quad \times\left(\frac{\|m_0-\overline{m}_0 \|_{B^{-1/2}_{2,\infty}}+\|n_0-\overline{n}_0 \|_{B^{-1/2}_{2,\infty}}}{C \hbar\left(4+4M^2\right)}\right)^{\exp{\left\{-C (1-\theta)\hbar^2\left(4+4M^2\right)\right\}}} ,
\end{split}
\end{equation}
which implies that the map $\Lambda(\cdot): X_{1/2,1}\rightarrow C_T(X_{-1/2,1})$ is H\"{o}lder continuous.

\textbf{Step 2:} Continuity of the data-to-solution map in $C_T(X_{1/2,1})$. Assume  $(m_{0,\infty},n_{0,\infty})\in X_{1/2,1}$ and $(m_{0,k},n_{0,k})$ tends to $(m_{0,\infty},n_{0,\infty})$ in $X_{1/2,1}$. We denote by $(m_k,n_k)$ the solution with respect to the initial data $(m_{0,k},n_{0,k})$. Note that the sequence $(m_{0,k},n_{0,k})$ taks values in a closed bounded ball $B_r(m_{0,\infty},n_{0,\infty})\subseteq X_{1/2,1}$ for some $r>0$, following the similar procedure in the first step, one can find $T,M>0$ such that for all $k\in\mathbb{N}$,
\begin{equation}\label{5.5}
\begin{split}
\sup_{t\in [0,T]}\Big(\|m_k(t)\|_{B_{2,1}^{1/2}} + \|n_k(t)\|_{B_{2,1}^{1/2}}\Big)\leq M.
\end{split}
\end{equation}
To prove that $(m_{k},n_{k})$ tends to $(m_{\infty},n_{\infty})$ in $C_T(X_{1/2,1})$, we shall follow the Kato's  method \cite{kato1975quasi} to decompose the solution of \eqref{P2NSQQ} as $(m_k,n_k)=(z_k^1,w_k^1)+(z_k^2,w_k^2)$ with
\begin{equation}\label{5.6}
\begin{cases}
\partial_t z_k^1+\rho_k \partial_x z_k^1=m_\infty\left(\psi_\infty(t,x)-\overline{\psi_\infty}(t)\right)-m_k\left(\psi_k(t,x)-\overline{\psi_k}(t)\right) ,\\

\partial_tw_k^1+\rho_k \partial_x w_k^1=n_\infty\left(\psi_\infty(t,x)-\overline{\psi_\infty}(t)\right)-n_k\left(\psi_k(t,x)-\overline{\psi_k}(t)\right) ,\\

z_k^1(x,0)= m_{0,k}(x)-m_{0,\infty}(x),\\
w_k^1(x,0)= n_{0,k}(x)-n_{0,\infty}(x),
\end{cases}
\end{equation}
and
\begin{equation}\label{5.7}
\begin{cases}
\partial_t z_k^2+\rho_k \partial_x z_k^2=-m_\infty\left(\psi_\infty(t,x)-\overline{\psi_\infty}(t)\right),\\

\partial_tw_k^2+\rho_k \partial_x w_k^2=-n_\infty\left(\psi_\infty(t,x)-\overline{\psi_\infty}(t)\right),\\

z_k^2(x,0)=  m_{0,\infty}(x),\\
w_k^2(x,0)=  n_{0,\infty}(x).
\end{cases}
\end{equation}
Using \eqref{5.5} and applying Lemma   \ref{lem:priorieti} to the first equation in  \eqref{5.6}, we get
\begin{equation}\label{5.8}
\begin{split}
 \|z_{k}^1(t)\|_{B^{ 1/2}_{2,1}}
 \leq & \exp{\left\{C\int_0^t\|\partial_x\rho_k\|_{B^{1/2}_{2,\infty}\cap L^\infty}dt'\right\}} \bigg(\|m_{0,k} -m_{0,\infty}\|_{B^{ 1/2}_{2,1}} \\
 &+C\int_0^t  \|m_\infty\left(\psi_\infty(t,x)-\overline{\psi_\infty}(t)\right)-m_k\left(\psi_k(t,x)-\overline{\psi_k}(t)\right)\|_{B^{ 1/2}_{2,1}}ds\bigg) .
\end{split}
\end{equation}
By using the fact that $B_{2,1}^{1/2} $ is a Banach space, we have  for any $k\in \overline{\mathbb{N}}$ and $t\in[0,T]$
\begin{equation}\label{5.9}
\begin{split}
\int_0^t\|\partial_x\rho_k\|_{B^{1/2}_{2,\infty}\cap L^\infty}dt' &\leq C \int_0^t\|(\alpha+\gamma)(v_k+\partial_x v_k)m_k-\alpha(u_k-\partial_x u_k)n_k\|_{B^{1/2}_{2,1}}dt'\\
&\leq C \int_0^t(|\alpha|+|\gamma|)(\|m_k\|_{B^{1/2}_{2,1}}+\|n_k\|_{B^{1/2}_{2,1}})^2dt' \leq C M^2.
\end{split}
\end{equation}
For the second term on the right hand side of \eqref{5.8}, we have
\begin{equation}\label{5.10}
\begin{split}
&\|m_\infty\left(\psi_\infty(t,x)-\overline{\psi_\infty}(t)\right)-m_k\left(\psi_k(t,x)-\overline{\psi_k}(t)\right)\|_{B^{ 1/2}_{2,1}}\\
&\quad \leq  \|(m_\infty-m_k)\left(\psi_\infty(t,x)-\overline{\psi_\infty}(t)\right)\|_{B^{ 1/2}_{2,1}}\\
&\quad\quad+\|m_k[\psi_\infty(t,x)-\psi_k(t,x)-(\overline{\psi_\infty}(t)-\overline{\psi_k}(t))]\|_{B^{ 1/2}_{2,1}} \doteq \mbox{I}_k(t)+\mbox{II}_k(t).
\end{split}
\end{equation}
For $\mbox{I}_k(t)$, we have
\begin{equation}\label{5.11}
\begin{split}
\mbox{I}_k(t) &\leq C \|m_k- m_\infty\|_{B^{ 1/2}_{2,1}}\Big(\| \psi_\infty(t,x)\|_{B^{ 1/2}_{2,1}}+\overline{\psi_\infty}(t) \|_{L^\infty}\Big)\\
&\leq C(|\alpha|+|\gamma| )\|m_k- m_\infty\|_{B^{ 1/2}_{2,1}}\Big(\|m_\infty\|_{B^{ 1/2}_{2,1}}+ \|n_\infty\|_{B^{ 1/2}_{2,1}}\Big)\\
&\leq CM(|\alpha|+|\gamma| ) \|m_k- m_\infty\|_{B^{ 1/2}_{2,1}}.
\end{split}
\end{equation}
For $\mbox{II}_k(t)$, we first note that
\begin{equation*}
\begin{split}
\psi_k(t,x)- \psi_\infty(t,x)= (\alpha+\gamma ) \varphi(t,x)+\alpha \phi(t,x),
\end{split}
\end{equation*}
where $\varphi(t,x)= [v_k-v_\infty+\partial_x(v_k-v_\infty)]m_k+(v_\infty+\partial_x v_\infty)(m_k-m_\infty)$, $
\phi(t,x)=[u_k-u_\infty-\partial_x(u_k-u_\infty)]n_k+(u_\infty-\partial_x u_\infty)(n_k-n_\infty)$. Then one can estimate $\mbox{II}_k(t)$ as follows
\begin{equation}\label{5.12}
\begin{split}
\mbox{II}_k(t)\leq& C \|m_k\|_{B^{ 1/2}_{2,1}}\Big(\|\psi_\infty(t,\cdot)-\psi_k(t,\cdot)\|_{B^{ 1/2}_{2,1}}+\|\psi_\infty(t,\cdot)-\psi_k(t,\cdot)\|_{L^\infty}\Big)\\
\leq& C (|\alpha|+|\gamma| )\|m_k\|_{B^{ 1/2}_{2,1}}\Big(\|v_k-v_\infty+\partial_x(v_k-v_\infty)\|_{B^{ 1/2}_{2,1}}\|m_k\|_{B^{ 1/2}_{2,1}}\\
&+\|v_\infty+\partial_x v_\infty\|_{B^{ 1/2}_{2,1}}\|m_k-m_\infty\|_{B^{ 1/2}_{2,1}}+\|u_\infty-\partial_x u_\infty\|_{B^{ 1/2}_{2,1}}\|n_k-n_\infty \|_{B^{ 1/2}_{2,1}}\\
&+  \|u_k-u_\infty-\partial_x(u_k-u_\infty)\|_{B^{ 1/2}_{2,1}}\|n_k\|_{B^{ 1/2}_{2,1}} \Big)\\
\leq& C (|\alpha|+|\gamma| )\|m_k\|_{B^{ 1/2}_{2,1}}\Big(\|n_k-n_\infty \|_{B^{1/2}_{2,1}}\|m_k\|_{B^{ 1/2}_{2,1}} +\|n_\infty \|_{B^{ 3/2}_{2,1}}\|m_k-m_\infty\|_{B^{ 1/2}_{2,1}}\\
&+\|m_\infty \|_{B^{3/2}_{2,1}}\|n_k-n_\infty \|_{B^{ 1/2}_{2,1}}+  \|m_k-m_\infty \|_{B^{1/2}_{2,1}}\|n_k\|_{B^{ 1/2}_{2,1}} \Big)\\
\leq& C M^2(|\alpha|+|\gamma| ) \Big(\|n_k-n_\infty \|_{B^{1/2}_{2,1}} +\|m_k-m_\infty \|_{B^{1/2}_{2,1}} \Big).
\end{split}
\end{equation}
Putting the estimates \eqref{5.9}-\eqref{5.12} together, we deduce that
\begin{equation*}
\begin{split}
 \|z_{k}^1(t)\|_{B^{ 1/2}_{2,1}}
 \leq & e^{CM^2} \left(\|m_{0,k} -m_{0,\infty}\|_{B^{ 1/2}_{2,1}} + \int_0^t(|\alpha|+|\gamma| )(\|n_k-n_\infty \|_{B^{1/2}_{2,1}} +\|m_k-m_\infty \|_{B^{1/2}_{2,1}} )dt'\right) .
\end{split}
\end{equation*}
In a similar manner, we have for the second equation in \eqref{5.6} that
\begin{equation*}
\begin{split}
 \|w_{k}^1(t)\|_{B^{ 1/2}_{2,1}}
 \leq & e^{CM^2} \left(\|n_{0,k} -n_{0,\infty}\|_{B^{ 1/2}_{2,1}} + \int_0^t(|\alpha|+|\gamma| )(\|n_k-n_\infty \|_{B^{1/2}_{2,1}} +\|m_k-m_\infty \|_{B^{1/2}_{2,1}} )dt'\right) .
\end{split}
\end{equation*}
Thereby we obtain
\begin{equation}\label{5.13}
\begin{split}
\|z_{k}^1(t)\|_{B^{ 1/2}_{2,1}}+ \|w_{k}^1(t)\|_{B^{ 1/2}_{2,1}}
 \leq & e^{CM^2} \bigg(\|m_{0,k} -m_{0,\infty}\|_{B^{ 1/2}_{2,1}} +\|n_{0,k} -n_{0,\infty}\|_{B^{ 1/2}_{2,1}} \\
 &+ \int_0^t(|\alpha|+|\gamma| )(\|n_k-n_\infty \|_{B^{1/2}_{2,1}} +\|m_k-m_\infty \|_{B^{1/2}_{2,1}} )dt'\bigg) .
\end{split}
\end{equation}

On the other hand, it is not difficult to verify that the terms on right hand side of \eqref{5.7} are uniformly bounded in $L^1_T(X_{1/2,1})$, and
 \begin{equation}\label{5.14}
\begin{split}
\sup_{k\geq 1}\|\partial_x \rho_k\|_{B_{2,1}^{1/2}}\leq C(|\alpha|+|\gamma| )\sup_{k\geq 1}\|m_k\|_{B^{ 1/2}_{2,1}}\|n_k\|_{B^{ 1/2}_{2,1}}\leq CM^2(|\alpha|+|\gamma| ).
\end{split}
\end{equation}
Moreover, by using the interpolation inequality
$\|f\|_{B^{\theta s_1+ (1-\theta)s_2}_{p,r}}\leq C
 \|f\|_{B^{s_1}_{p,r}}^\theta \|f\|_{B^{s_2}_{p,r}}^{1-\theta}
$,  for any $\theta\in (0,1)$ (cf. \cite{chemin2004localization,64}), one can estimate the difference of $\rho_k$ and $\rho_\infty$ by
 \begin{eqnarray}\label{5.15}
\begin{split}
 \int_0^T\|\rho_k(t)-\rho_\infty(t)\|_{B_{2,1}^{1/2}}dt\leq & \int_0^T\Big(\|[v_k-v_\infty+\partial_x(v_k-v_\infty)]m_k+(v_\infty+\partial_x v_\infty)(m_k-m_\infty)\|_{B_{2,1}^{-1/2}}\\
&+\|[u_k-u_\infty-\partial_x(u_k-u_\infty)]n_k+(u_\infty-\partial_x u_\infty)(n_k-n_\infty)\|_{B_{2,1}^{-1/2}}\Big)dt\\
\leq & \int_0^T\bigg[ (\|v_k-v_\infty\|_{B_{2,1}^{-1/2}}+\|\partial_x(v_k-v_\infty)\|_{B_{2,1}^{-1/2}} )^{1-\theta}\|m_k\|_{B_{2,1}^{ 1/2}}^\theta\\
&+ (\|v_\infty\|_{B_{2,1}^{-1/2}}+\|\partial_x v_\infty\|_{B_{2,1}^{1/2}} )^\theta\|m_k-m_\infty\|_{B_{2,1}^{-1/2}}^{1-\theta}\\
&+ (\|u_k-u_\infty\|_{B_{2,1}^{-1/2}}+\|\partial_x(u_k-u_\infty)\|_{B_{2,1}^{-1/2}} )^{1-\theta}\|n_k\|_{B_{2,1}^{1/2}}^\theta\\
&+ (\|u_\infty\|_{B_{2,1}^{1/2}}+\|\partial_x u_\infty\|_{B_{2,1}^{1/2}} )^\theta\|n_k-n_\infty\|_{B_{2,1}^{-1/2}} ^{1-\theta}\bigg]dt\\
\leq & CM ^\theta\int_0^T\Big( \|m_k-m_\infty\|_{B_{2,1}^{-1/2}}^{1-\theta} +\|n_k-n_\infty\|_{B_{2,1}^{-1/2}} ^{1-\theta} \Big)dt.
\end{split}
\end{eqnarray}
As $(m_{k},n_{k})$ tends to $(m_{\infty},n_{\infty})$ strongly in $C_T(X_{-1/2,1})$, the last inequality implies that $\rho_k$ tends to $\rho_\infty$ strongly in $L^1_T(B_{2,1}^{1/2})$.  To deal with the convergence of the system \eqref{5.7}, we recall the following useful lemma which was firstly proposed by Danchin.
\begin{lemma} [\cite{danchin2003note}] \label{danchin}
Denote $\overline{\mathbb{N}}= \mathbb{N}\cup \infty$. Let $(v_k)_{k\in \overline{\mathbb{N}}}$ be a sequence of functions in $C_T(B_{2,1}^{1/2})$. Assume that $v_k$ solves the following equation
\begin{equation*}
\begin{cases}
 \partial_t v_k+a_k\partial_xv_k=f,\\
 v_k(x,0)=v_0
\end{cases}
\end{equation*}
with $v_0\in B_{2,1}^{1/2}$, $f\in L^1_T(B_{2,1}^{1/2})$ and that $\sup_{k\geq 1} \|\partial_x a_k(t)\|_{B_{2,1}^{1/2}}\leq \beta (t)$ for some $\beta\in L^1(0,T)$. If in addition $a_k$ tends to $a_\infty$ in $L^1_T(B_{2,1}^{1/2})$, then $v_k$ tends to $v_\infty$ in $C_T(B_{2,1}^{1/2})$.
\end{lemma}
It then follows from \eqref{5.14}, \eqref{5.15} and Lemma   \ref{danchin} that
 \begin{equation}\label{5.16}
\begin{split}
(z_k^2,w_k^2)\stackrel{k\rightarrow\infty}{\longrightarrow}  (m_{\infty},n_{\infty})\quad  \mbox{strongly in}\quad C_T(X_{1/2,1}).
\end{split}
\end{equation}
By the decomposition of $(m_k,n_k)$ we have
 \begin{equation*}
\begin{split}
 \|m_k-m_\infty \|_{B^{1/2}_{2,1}}+\|n_k-n_\infty \|_{B^{1/2}_{2,1}}  \leq \|z_k^1\|_{B^{1/2}_{2,1}}+\|w_k^1\|_{B^{1/2}_{2,1}}+\|z_k^2-m_\infty \|_{B^{1/2}_{2,1}} +\|w_k^2-n_\infty \|_{B^{1/2}_{2,1}},
\end{split}
\end{equation*}
it then follows from \eqref{5.13} that
\begin{equation*}
\begin{split}
&\|z_{k}^1(t)\|_{B^{ 1/2}_{2,1}}+ \|w_{k}^1(t)\|_{B^{ 1/2}_{2,1}}\\
 &\quad \leq e^{CM^2} \bigg(\|m_{0,k} -m_{0,\infty}\|_{B^{ 1/2}_{2,1}} +\|n_{0,k} -n_{0,\infty}\|_{B^{ 1/2}_{2,1}}+ \int_0^t(|\alpha|+|\gamma| )(\|z_k^1\|_{B^{1/2}_{2,1}}+\|w_k^1\|_{B^{1/2}_{2,1}})dt' \\
 &\quad\quad+\int_0^t(|\alpha|+|\gamma| )(\|z_k^2-m_\infty \|_{B^{1/2}_{2,1}} +\|w_k^2-n_\infty \|_{B^{1/2}_{2,1}} )dt'\bigg),
\end{split}
\end{equation*}
 which together with the Gronwall's lemma yield that
 \begin{equation*}
\begin{split}
&\|z_{k}^1(t)\|_{B^{ 1/2}_{2,1}}+ \|w_{k}^1(t)\|_{B^{ 1/2}_{2,1}}\leq e^{e^{CM^2}} \bigg(\|m_{0,k} -m_{0,\infty}\|_{B^{ 1/2}_{2,1}} +\|n_{0,k} -n_{0,\infty}\|_{B^{ 1/2}_{2,1}})\\
 &\quad + \int_0^t (|\alpha|+|\gamma| )(\|z_k^2-m_\infty \|_{B^{1/2}_{2,1}} +\|w_k^2-n_\infty \|_{B^{1/2}_{2,1}})dt'\bigg),
\end{split}
\end{equation*}
which implies that
\begin{equation}\label{5.17}
\begin{split}
(z_k^1,w_k^1)\stackrel{k\rightarrow\infty}{\longrightarrow}  (0,0)\quad  \mbox{strongly in}\quad C_T(X_{1/2,1}).
\end{split}
\end{equation}
From \eqref{5.16} and \eqref{5.17}, we get
 \begin{equation*}
\begin{split}
(m_k,n_k)\stackrel{k\rightarrow\infty}{\longrightarrow}  (m_{\infty},n_{\infty})\quad  \mbox{strongly in}\quad C_T(X_{1/2,1}).
\end{split}
\end{equation*}
Combining the first step and second step, we completed the proof of Theorem \ref{theorem1}.
\end{proof}

\section{blow-up phenomena}
The aim of this section is to prove the blow-up criteria stated in Theorem \ref{theorem2} and Theorem \ref{theorem3} with initial data possessing different regularity.

\begin{proof}[Proof of Theorem \ref{theorem2}] Applying the nonhomogeneous dyadic blocks $\Delta_q$ to the first component with respect to $m$ in \eqref{P2NSQQ}, we get
\begin{equation}\label{6.1}
\begin{split}
\partial_t \Delta_q m+ \rho \partial_x \Delta_q m=[\rho,\Delta_q]\partial_x m-\Delta_q[m\left(\psi(t,x)-\overline{\psi}(t)\right)].
\end{split}
\end{equation}
Multiplying  \eqref{6.1} by $2\Delta_q m$ and integrating in space, we have
\begin{equation*}
\begin{split}
 \frac{d}{dt}\int_\mathbb{T}|\Delta_q m|^2dx =&\int_\mathbb{T} \partial_x\rho (\Delta_q m)^2 dx+ 2\int_\mathbb{T}\Delta_q m [\rho,\Delta_q]\partial_x m dx\\
 &-2\int_\mathbb{T}\Delta_q m \Delta_q[m\left(\psi(t,x)-\overline{\psi}(t)\right)] dx\\
\leq  &  \|\partial_x\rho\|_{L^\infty}\|\Delta_q m\|_{L^2}^2+ 2\|\Delta_q m\|_{L^2}\|[\rho\partial_x,\Delta_q] m\|_{L^2}\\
 &+2\|\Delta_q m\|_{L^2} \|\Delta_q[m\left(\psi(t,x)-\overline{\psi}(t)\right)]\|_{L^2},
\end{split}
\end{equation*}
which implies that
\begin{equation}\label{6.2}
\begin{split}
 \frac{d}{dt}\|\Delta_q m\|_{L^2}\leq  \frac{1}{2} \|\partial_x\rho\|_{L^\infty}\|\Delta_q m\|_{L^2}+  \|[\rho\partial_x,\Delta_q] m\|_{L^2}+ \|\Delta_q[m\left(\psi(t,x)-\overline{\psi}(t)\right)]\|_{L^2}.
\end{split}
\end{equation}
Multiplying \eqref{6.2} by $2^{q/2}$ and taking the $l^1$-norm for $q\in \mathbb{Z}$, we obtain
\begin{equation}\label{6.3}
\begin{split}
 \| m(t)\|_{B_{2,1}^{1/2}}\leq&  \| m_0\|_{B_{2,1}^{1/2}}+\frac{1}{2} \int_0^t\|\partial_x\rho\|_{L^\infty}\| m\|_{B_{2,1}^{1/2}}dt'+   \int_0^t\sum_{q\geq -1}2^{q/2}\|[\rho\partial_x,\Delta_q] m\|_{L^2}dt'\\
 &+  \int_0^t\|m\left(\psi(t,x)-\overline{\psi}(t)\right)\|_{B_{2,1}^{1/2}}dt'.
\end{split}
\end{equation}
In view of the definition of $\rho(t,x)$, we have
\begin{equation}\label{6.4}
\begin{split}
\|\partial_x\rho(t)\|_{L^\infty} &= \| \psi(t,\cdot)-\overline{\psi}(t)\|_{L^\infty}\\
&\leq C\|(\alpha+\gamma)(v+v_x)m-\alpha(u-u_x)n\|_{L^\infty} \leq C(|\alpha|+|\gamma| )\|m\|_{L^\infty}\|n\|_{L^\infty}.
\end{split}
\end{equation}
To estimate the third term on the right hand side of \eqref{6.3}, we need the following lemma.
\begin{lemma} [\cite{64}] \label{commutator}
Let $0<\sigma<1 $, $1\leq  r\leq\infty$, $1\leq p \leq p_1 \leq\infty$. If $v$ be a vector field over $\mathbb{R}^d$, then there exists a constant $C$ such that
$$
\|(2^{j\sigma}\|[v\cdot \nabla,\Delta_j]f\|_{L^p})_j\|_{l^r}\leq C\|\nabla v\|_{L^\infty}\|f\|_{B_{p,r}^\sigma}.
$$
\end{lemma}
It follows from Lemma \ref{commutator} with $\sigma=1/2$ that
\begin{equation}\label{6.5}
\begin{split}
\int_0^t\sum_{q\geq -1}2^{q/2}\|[\rho\partial_x,\Delta_q] m\|_{L^2}dt'&\leq  C\int_0^t\|\partial_x\rho\|_{L^\infty}\|m\|_{B_{2,1}^{1/2}}dt'\\
&\leq  C\int_0^t(|\alpha|+|\gamma| )\|m\|_{L^\infty}\|n\|_{L^\infty}\|m\|_{B_{2,1}^{1/2}}dt'.
\end{split}
\end{equation}
Using the Moser-type estimate $\|fg\|_{B_{p,r}^s}\leq C (\|f\|_{B_{p,r}^s}\|g\|_{L^\infty}+\|f\|_{L^\infty}\|g\|_{B_{p,r}^s})$ for any $s>0$, we have
\begin{equation}\label{6.6}
\begin{split}
&\int_0^t\|m\left(\psi(t,x)-\overline{\psi}(t)\right)\|_{B_{2,1}^{1/2}}dt'\\
&\quad \leq C\int_0^t\Big(\|m\|_{B_{2,1}^{1/2}}\|\psi(t,x)-\overline{\psi}(t)\|_{L^\infty}
+\|m\|_{L^\infty}\|\psi(t,x)-\overline{\psi}(t)\|_{B_{2,1}^{1/2}}\Big)dt'\\
&\quad \leq C\int_0^t(|\alpha|+|\gamma|)\Big[\|m\|_{B_{2,1}^{1/2}}\|m\|_{L^\infty}\|n\|_{L^\infty}
+\|m\|_{L^\infty}\Big(\|v+v_x\|_{L^\infty}\|m\|_{B_{2,1}^{1/2}}\\
&\qquad +\|m\|_{L^\infty}\|v+v_x\|_{B_{2,1}^{1/2}}+\|u-u_x\|_{L^\infty}\|n\|_{B_{2,1}^{1/2}}
+\|n\|_{L^\infty}\|u-u_x\|_{B_{2,1}^{1/2}}\Big)\Big]dt'\\
&\quad \leq C\int_0^t(|\alpha|+|\gamma|)\Big( \|m\|_{L^\infty}\|n\|_{L^\infty}\|m\|_{B_{2,1}^{1/2}} +\|m\|_{L^\infty}^2\|n\|_{B_{2,1}^{1/2}}\Big)dt'.
\end{split}
\end{equation}
Putting the estimates \eqref{6.4}-\eqref{6.6} into \eqref{6.3}, we get
\begin{equation*}
\begin{split}
 \| m(t)\|_{B_{2,1}^{1/2}}\leq&  \| m_0\|_{B_{2,1}^{1/2}} +  C\int_0^t(|\alpha|+|\gamma|)\Big( \|m\|_{L^\infty}\|n\|_{L^\infty}\|m\|_{B_{2,1}^{1/2}} +\|m\|_{L^\infty}^2\|n\|_{B_{2,1}^{1/2}} \Big)dt'.
\end{split}
\end{equation*}
Similar to the last estimate, one can also derive that
\begin{equation*}
\begin{split}
 \| n(t)\|_{B_{2,1}^{1/2}}\leq&  \| n_0\|_{B_{2,1}^{1/2}} +  C\int_0^t(|\alpha|+|\gamma|)\Big( \|m\|_{L^\infty}\|n\|_{L^\infty}\|n\|_{B_{2,1}^{1/2}} +\|n\|_{L^\infty}^2\|m\|_{B_{2,1}^{1/2}} \Big)dt'.
\end{split}
\end{equation*}
As a result, we get from the last two inequality and the Sobolev embedding $\dot{B}_{\infty,1}^0\hookrightarrow L^\infty$ that
\begin{equation}\label{6.7}
\begin{split}
 &\| m(t)\|_{B_{2,1}^{1/2}}+ \| n(t)\|_{B_{2,1}^{1/2}}\leq   \|m_0\|_{B_{2,1}^{1/2}}+\| n_0\|_{B_{2,1}^{1/2}}\\
 &\qquad +  C\int_0^t (|\alpha|+|\gamma|)\Big(\|m\|_{\dot{B}_{\infty,1}^0}^2+\|n\|_{\dot{B}_{\infty,1}^0}^2\Big)\Big(\| m(t)\|_{B_{2,1}^{1/2}}+ \| n(t)\|_{B_{2,1}^{1/2}}\Big)dt'.
\end{split}
\end{equation}
An application of the Gronwall's inequality to \eqref{6.7} leads to
\begin{equation}\label{6.8}
\begin{split}
 \| m(t)\|_{B_{2,1}^{1/2}}+ \| n(t)\|_{B_{2,1}^{1/2}}\leq \Big(\|m_0\|_{B_{2,1}^{1/2}}+\| n_0\|_{B_{2,1}^{1/2}}\Big) \exp{\left\{C\int_0^t (|\alpha|+|\gamma|)(\|m\|_{\dot{B}_{\infty,1}^0}^2+\|n\|_{\dot{B}_{\infty,1}^0}^2)dt'\right\}}.
\end{split}
\end{equation}
Assume that there exists a constant $C>0$ such that
$$
\int_0^{T^*} (|\alpha|+|\gamma|)\Big(\|m\|_{\dot{B}_{\infty,1}^0}^2+\|n\|_{\dot{B}_{\infty,1}^0}^2\Big) dt'\leq C.
$$
It follows from \eqref{6.8} that
\begin{equation}\label{6.9}
\begin{split}
 \limsup_{t\rightarrow T^*}\Big(\| m(t)\|_{B_{2,1}^{1/2}}+ \| n(t)\|_{B_{2,1}^{1/2}}\Big) \leq C\Big(\|m_0\|_{B_{2,1}^{1/2}}+\| n_0\|_{B_{2,1}^{1/2}}\Big).
\end{split}
\end{equation}
According to Theorem \ref{theorem1}, given $(m(T^*,x),n(T^*,x))$, there exists a $\delta>0$ such that the solution $(m,n)$ can be extended to the interval $[T^*,T^*+\delta]$, which contradicts to the fact that $T^*$ is the maximum existence time.

Now let us derive the lower bound for the lifespan $T^*$. Using the embedding $B_{2,1}^{1/2}\hookrightarrow L^\infty$, it follows from \eqref{6.7} that
\begin{equation}\label{6.10}
\begin{split}
 &\| m(t)\|_{B_{2,1}^{1/2}}+ \| n(t)\|_{B_{2,1}^{1/2}}\\
 &\quad \leq   \|m_0\|_{B_{2,1}^{1/2}}+\| n_0\|_{B_{2,1}^{1/2}} +  C\int_0^t (|\alpha|+|\gamma|)\Big(\| m(t')\|_{B_{2,1}^{1/2}}+ \| n(t')\|_{B_{2,1}^{1/2}}\Big)^3dt'.
\end{split}
\end{equation}
Solving the inequality \eqref{6.10} leads to
\begin{equation}\label{6.11}
\begin{split}
  \| m(t)\|_{B_{2,1}^{1/2}}+ \| n(t)\|_{B_{2,1}^{1/2}} \leq \frac{\| m_0\|_{B_{2,1}^{1/2}}+ \| n_0\|_{B_{2,1}^{1/2}}}{\left[1-C\Big(\| m_0\|_{B_{2,1}^{1/2}}+ \| n_0\|_{B_{2,1}^{1/2}}\Big)^2\int_0^t(|\alpha|+|\gamma|)dt'\right]^{1/2}}.
\end{split}
\end{equation}
By the assumption of $\alpha,\gamma\in L^1_{loc}(0,\infty;\mathbb{R})$, the indefinite integral on any finite $[0,t]$ is absolutely continuous, hence one can find a $T(m_0,n_0)$ defined in Theorem \ref{theorem2} such that the blow-up time satisfies $ T^*\geq T(m_0,n_0)$.

The proof of Theorem \ref{theorem2} is completed.
\end{proof}

\begin{proof} [Proof of Theorem \ref{theorem3}]
 We first recall the Brezis-Gallouet-Wainger type estimate (cf. \cite{kozono2002critical})
\begin{equation}\label{6.12}
\begin{split}
  \|f\|_{L^\infty}\leq C\left(1+\|f\|_{\dot{B}_{p,\rho}^{1/p}} \left(\log(e+\|f\|_{B_{q,\sigma}^s})\right)^{1-1/\rho}\right),\quad \forall f \in \dot{B}_{p,\rho}^{1/p}\cap B_{q,\sigma}^s,
\end{split}
\end{equation}
where $q\in[1,\infty)$, $p,\rho,\sigma\in[1,\infty]$ and $s>1/q$.

Assume that $(m_0,n_0)\in X_{1/2+\epsilon,r}$, for any $\epsilon\in (0,1/2)$ and $r\in [1,\infty]$, Theorem \ref{theorem2} in \cite{zhang2020periodic} implies that the Cauchy problem \eqref{P2NSQQ}
admits an unique solution $(m,n)\in C([0,T];X_{1/2+\epsilon,r})$ for some time $T>0$.

Multiplying both sides of \eqref{6.2} by $2^{qr/2}$ and then taking the $l^r$-norm for $q\in \mathbb{Z}$, we obtain
\begin{equation}\label{6.13}
\begin{split}
 \frac{d}{dt}\| m(t)\|_{B_{2,r}^{1/2+\epsilon}}\leq&  \frac{1}{2} \|\partial_x\rho\|_{L^\infty}\| m\|_{B_{2,r}^{1/2+\epsilon}}+   \left(\sum_{q\geq -1}2^{qr(1/2+\epsilon) }\|[\rho\partial_x,\Delta_q] m\|_{L^2}^r\right)^{1/r}\\
 &+ \|m\left(\psi(t,x)-\overline{\psi}(t)\right)\|_{B_{2,r}^{1/2+\epsilon}}.
\end{split}
\end{equation}
By using the commutator estimate in Lemma \ref{commutator}, we have
\begin{equation*}
\begin{split}
 \left(\sum_{q\geq -1}2^{qr(1/2+\epsilon)}\|[\rho\partial_x,\Delta_q] m\|_{L^2}^r\right)^{1/r} \leq C(|\alpha|+|\gamma| )\|m\|_{L^\infty}\|n\|_{L^\infty}\|m\|_{B_{2,1}^{1/2+\epsilon}}.
\end{split}
\end{equation*}
For the third term on the right hand side of \eqref{6.14}, we have
\begin{equation*}
\begin{split}
&\|m\left(\psi(t,x)-\overline{\psi}(t)\right)\|_{B_{2,1}^{1/2+\epsilon}}\leq C (|\alpha|+|\gamma|) \Big( \|m\|_{L^\infty}\|n\|_{L^\infty}\|m\|_{B_{2,1}^{1/2+\epsilon}} +\|m\|_{L^\infty}^2\|n\|_{B_{2,1}^{1/2+\epsilon}} \Big).
\end{split}
\end{equation*}
Putting the last two estimates into \eqref{6.13} leads to
\begin{equation*}
\begin{split}
 \frac{d}{dt}\| m(t)\|_{B_{2,r}^{1/2+\epsilon}}\leq  C (|\alpha|+|\gamma|) \Big( \|m\|_{L^\infty}\|n\|_{L^\infty}\|m\|_{B_{2,1}^{1/2+\epsilon}} +\|m\|_{L^\infty}^2\|n\|_{B_{2,1}^{1/2+\epsilon}}\Big).
\end{split}
\end{equation*}
Applying the similar argument to the equation with respect to the solution $n$, we have
\begin{equation*}
\begin{split}
 \frac{d}{dt}\| n(t)\|_{B_{2,r}^{1/2+\epsilon}}\leq  C (|\alpha|+|\gamma|) \Big( \|m\|_{L^\infty}\|n\|_{L^\infty}\|n\|_{B_{2,1}^{1/2+\epsilon}} +\|n\|_{L^\infty}^2\|m\|_{B_{2,1}^{1/2+\epsilon}}\Big).
\end{split}
\end{equation*}
It then follows that
\begin{equation}\label{6.14}
\begin{split}
 &\frac{d}{dt}\Big(\| m(t)\|_{B_{2,r}^{1/2+\epsilon}}+\| n(t)\|_{B_{2,r}^{1/2+\epsilon}}\Big)\\
 &\quad \leq  C (|\alpha|+|\gamma|) \left(\|m\|_{L^\infty}^2+\|n\|_{L^\infty}^2\right)\Big(\| m(t)\|_{B_{2,r}^{1/2+\epsilon}}+\| n(t)\|_{B_{2,r}^{1/2+\epsilon}}\Big)  .
\end{split}
\end{equation}
In terms of the logarithmic Sobolev inequality in Lemma \ref{logestimate}, we get
\begin{equation}\label{6.15}
\begin{split}
  \|m\|_{L^\infty}^2+  \|n\|_{L^\infty}^2&\leq C \Big(1+\|m\|_{\dot{B}_{\infty,2}^{0}}^2+\|n\|_{\dot{B}_{\infty,2}^{0}}^2\Big) \Big(\log(e+\|m\|_{B_{2,r}^{1/2+\epsilon}})+\log(e+\|m\|_{B_{2,r}^{1/2+\epsilon}})\Big) \\
  &\leq C \Big(1+\|m\|_{\dot{B}_{\infty,2}^{0}}^2+\|n\|_{\dot{B}_{\infty,2}^{0}}^2\Big) \log\Big(2e^2+\|m\|_{B_{2,r}^{1/2+\epsilon}}^2+\|n\|_{B_{2,r}^{1/2+\epsilon}}^2\Big).
\end{split}
\end{equation}
Inserting the estimate \eqref{6.15} into \eqref{6.14} leads to
\begin{equation}\label{6.16}
\begin{split}
 &\frac{d}{dt}\Big(\| m(t)\|_{B_{2,r}^{1/2+\epsilon}}+\| n(t)\|_{B_{2,r}^{1/2+\epsilon}}\Big)\\
 &\quad \leq  C (|\alpha|+|\gamma|) \Big(1+\|m\|_{\dot{B}_{\infty,2}^{0}}^2+\|n\|_{\dot{B}_{\infty,2}^{0}}^2\Big) \log\Big(2e^2+\|m\|_{B_{2,r}^{1/2+\epsilon}}^2+\|n\|_{B_{2,r}^{1/2+\epsilon}}^2\Big)\\
 &\quad\quad\times \Big(\| m(t)\|_{B_{2,r}^{1/2+\epsilon}}+\| n(t)\|_{B_{2,r}^{1/2+\epsilon}}\Big)  .
\end{split}
\end{equation}
Multiplying both sides of \eqref{6.16} by $\| m(t)\|_{B_{2,r}^{1/2+\epsilon}}+\| n(t)\|_{B_{2,r}^{1/2+\epsilon}}$, and using the fact of $(a+b)^2\approx a^2+b^2$, the inequality \eqref{6.16} is equivalent to
\begin{equation}\label{6.17}
\begin{split}
 &\frac{d}{dt}\Big(2e^2+\| m(t)\|_{B_{2,r}^{1/2+\epsilon}}^2+\| n(t)\|_{B_{2,r}^{1/2+\epsilon}}^2\Big)\\
 &\quad \leq  C (|\alpha|+|\gamma|) \Big(1+\|m\|_{\dot{B}_{\infty,2}^{0}}^2+\|n\|_{\dot{B}_{\infty,2}^{0}}^2\Big) \log\Big(2e^2+\|m\|_{B_{2,r}^{1/2+\epsilon}}^2+\|n\|_{B_{2,r}^{1/2+\epsilon}}^2\Big)\\
 &\quad\quad\times \Big(2e^2+\| m(t)\|_{B_{2,r}^{1/2+\epsilon}}^2+\| n(t)\|_{B_{2,r}^{1/2+\epsilon}}^2\Big)  .
\end{split}
\end{equation}
Solving the inequality \eqref{6.17} leads to
\begin{equation*}
\begin{split}
& \log\Big(2e^2+\|m\|_{B_{2,r}^{1/2+\epsilon}}^2+\|n\|_{B_{2,r}^{1/2+\epsilon}}^2\Big)\\
 &  \quad\leq \log\Big(2e^2+\|m_0\|_{B_{2,r}^{1/2+\epsilon}}^2+\|n_0\|_{B_{2,r}^{1/2+\epsilon}}^2\Big)\exp{\left\{C\int_0^t (|\alpha|+|\gamma|) (1+\|m\|_{\dot{B}_{\infty,2}^{0}}^2+\|n\|_{\dot{B}_{\infty,2}^{0}}^2) dt'\right\}}.
\end{split}
\end{equation*}
With the above estimate at hand, by using the similar argument  as that in the proof of Theorem \ref{theorem2}, one can obtain the blow-up criteria presented in Theorem \ref{theorem3}.

To derive the lower bound for the blow-up time $T^*$, we shall use a method which is different from \eqref{6.10}. More specifically, by applying the logarithmic Sobolev inequality \eqref{6.15}, we have
\begin{eqnarray}\label{6.18}
\begin{split}
  \|m\|_{L^\infty}^2+ \|n\|_{L^\infty}^2&\leq C\bigg[\Big(1+\|m\|_{\dot{B}_{\infty,\infty}^{0}} \log(e+\|m\|_{B_{2,r}^{1/2+\epsilon}}) \Big)^2+\Big(1+\|n\|_{\dot{B}_{\infty,\infty}^{0}} \log(e+\|n\|_{B_{2,r}^{1/2+\epsilon}}) \Big)^2\bigg]\\
  &\leq C\Big(1+\|m\|_{\dot{B}_{\infty,\infty}^{0}}^2 +\|n\|_{\dot{B}_{\infty,\infty}^{0}}^2\Big)\Big(\log(2e^2+\|m\|_{B_{2,r}^{1/2+\epsilon}}^2+\|n\|_{B_{2,r}^{1/2+\epsilon}}^2)\Big)^2.
\end{split}
\end{eqnarray}

Note that the existence of the quadratic function $\log^2\Big(2e^2+\|m\|_{B_{2,r}^{1/2+\epsilon}}^2+\|n\|_{B_{2,r}^{1/2+\epsilon}}^2\Big)$ in \eqref{6.18} makes the combination of estimates \eqref{6.14} and \eqref{6.18} insufficient for the establishment of the blow-up criteria in space $\dot{B}_{\infty,\infty}^{0}$. Further, in terms of the Sobolev embedding $B_{2,r}^{1/2+\epsilon}\hookrightarrow B_{\infty,\infty}^{0}\hookrightarrow\dot{B}_{\infty,\infty}^{0}$,  we get from \eqref{6.18} that
\begin{equation}\label{6.188}
\begin{split}
  \|m\|_{L^\infty}^2+ \|n\|_{L^\infty}^2 \leq C\Big(1+\|m\|_{B_{2,r}^{1/2+\epsilon}}^2 +\|n\|_{B_{2,r}^{1/2+\epsilon}}^2\Big)\Big(\log(2e^2+\|m\|_{B_{2,r}^{1/2+\epsilon}}^2+\|n\|_{B_{2,r}^{1/2+\epsilon}}^2)\Big)^2.
\end{split}
\end{equation}
Combining the estimate \eqref{6.14} and \eqref{6.188}  and using the inequality $\log(2e^2+x)\leq 2e^2+x$ for all $x\geq 0$,
we have
\begin{equation}\label{6.19}
\begin{split}
 &\frac{d}{dt}\Big(\| m(t)\|_{B_{2,r}^{1/2+\epsilon}}+\| n(t)\|_{B_{2,r}^{1/2+\epsilon}}\Big)\\
 &\quad \leq  C (|\alpha|+|\gamma|) \Big(1+\|m\|_{B_{2,r}^{1/2+\epsilon}}^2 +\|n\|_{B_{2,r}^{1/2+\epsilon}}^2\Big) \Big(\log(2e^2+\|m\|_{B_{2,r}^{1/2+\epsilon}}^2+\|n\|_{B_{2,r}^{1/2+\epsilon}}^2)\Big)^2\\
 &\quad\quad\times\Big(\| m(t)\|_{B_{2,r}^{1/2+\epsilon}}+\| n(t)\|_{B_{2,r}^{1/2+\epsilon}}\Big) \\
 &\quad \leq  C (|\alpha|+|\gamma|) \Big(2e^2+\|m\|_{B_{2,r}^{1/2+\epsilon}}^2+\|n\|_{B_{2,r}^{1/2+\epsilon}}^2\Big)^3\Big(\| m(t)\|_{B_{2,r}^{1/2+\epsilon}}+\| n(t)\|_{B_{2,r}^{1/2+\epsilon}}\Big)\\
 &\quad \leq  C (|\alpha|+|\gamma|) \Big(\sqrt{2}e+\|m\|_{B_{2,r}^{1/2+\epsilon}}+\|n\|_{B_{2,r}^{1/2+\epsilon}}\Big)^7 .
\end{split}
\end{equation}
Solving the inequality \eqref{6.19} leads to
\begin{equation*}
\begin{split}
&\sqrt{2}e+\| m(t)\|_{B_{2,r}^{1/2+\epsilon}}+\| n(t)\|_{B_{2,r}^{1/2+\epsilon}} \\
 &\quad \leq \frac{\sqrt{2}e+\| m_0\|_{B_{2,r}^{1/2+\epsilon}}+\| n_0\|_{B_{2,r}^{1/2+\epsilon}} }{\left[1-C\int_0^t(|\alpha|+|\gamma|)dt'\Big(\sqrt{2}e+\| m_0\|_{B_{2,r}^{1/2+\epsilon}}+\| n_0\|_{B_{2,r}^{1/2+\epsilon}} \Big)^6\right]^{1/6}}.
\end{split}
\end{equation*}
Due to the condition of $\alpha,\gamma\in L^1_{loc}(0,\infty;\mathbb{R})$, the indefinite integral on any finite interval $[0,t]$ is absolutely continuous, so there is a constant $T'(m_0,n_0)>0$ defined in Theorem \ref{theorem3} such that the lifespan $T^*$ of the solution $(m,n)$ satisfies $T^*\geq T'(m_0,n_0)$.

This completes the proof of Theorem \ref{theorem3}.
\end{proof}

\noindent
\textbf{Acknowledgements.}  The authors thank the National Natural Science Foundation of China (Project \# 11701198, \# 11971185 and \# 11971475), the Fundamental Research Funds for the Central Universities  (Project \# 5003011025),  and the 2019 - 2020 Hunan overseas distinguished professorship (Project \# 2019014) for their partial support. The author Qiao also thanks the UT President's Endowed Professorship (Project \# 450000123) for its partial support.

\bibliographystyle{plain}%
\bibliography{zhanglei2020}

\end{document}